\documentclass[11pt,a4paper]{amsart}

\usepackage[utf8]{inputenc}
\usepackage[T1]{fontenc} 
\usepackage[normalem]{ulem}
\input xy
\xyoption{all}
\usepackage[all,cmtip]{xy}
\usepackage{tikz-cd}
\usepackage{amsfonts}
\usepackage{amsthm}
\usepackage{amssymb}
\usepackage{mathrsfs}
\usepackage{amsmath}  
\usepackage{amscd}

\usepackage[svgnames,x11names,dvipsnames]{xcolor}
\usepackage[colorlinks,linkcolor=Red3,citecolor=Red3,urlcolor=Cyan3]{hyperref}
\usepackage{enumitem}
\setlist[itemize,1]{leftmargin=*,wide}
\setlist[enumerate,1]{leftmargin=*,wide,label=\upshape\bfseries\arabic*.}
\usepackage{xstring}
\usepackage[abspath]{currfile}
\usepackage[yyyymmdd]{datetime}

\newtheorem{theorem}{Theorem}[section]

\newtheorem{lemma}[theorem]{Lemma}
\newtheorem{proposition}[theorem]{Proposition}
\theoremstyle{definition}
\newtheorem{definition}[theorem]{Definition}
\newtheorem{example}[theorem]{Example}
\theoremstyle{remark}
\newtheorem{remark}[theorem]{Remark}

\numberwithin{equation}{section}

\newcommand{\GL}{\operatorname{Gl}}
\DeclareMathOperator{\Spec}{Spec}

\DeclareMathOperator{\Iso}{Iso}
\DeclareMathOperator{\Res}{Res}
\DeclareMathOperator{\Ker}{Ker}
\DeclareMathOperator{\Dim}{dim}

\DeclareMathOperator{\rk}{wt_{rk}}

\DeclareMathOperator{\Aut}{Aut}
\DeclareMathOperator{\diag}{diag}

\DeclareMathOperator{\ord}{ord}

\newcommand{\Mat}{\operatorname{Mat}}
\newcommand{\wt}{\operatorname{wt}}
\newcommand{\fm}{\mathfrak{m}}
\newcommand{\ev}{\operatorname{ev}}

\newcommand{\Supp}{\operatorname{Supp}}

\newcommand{\im}{\operatorname{Im}}
\newcommand{\F}{\mathbb{F}}

\makeatletter
\newcommand*{\doublerightarrow}[2]{\mathrel{
		\settowidth{\@tempdima}{$\scriptstyle#1$}
		\settowidth{\@tempdimb}{$\scriptstyle#2$}
		\ifdim\@tempdimb>\@tempdima \@tempdima=\@tempdimb\fi
		\mathop{\vcenter{
				\offinterlineskip\ialign{\hbox to\dimexpr\@tempdima+1em{##}\cr
					\rightarrowfill\cr\noalign{\textrm{Ker}n.5ex}
					\rightarrowfill\cr}}}\limits^{\!#1}_{\!#2}}}

\begin{document}

\title{Differential Goppa codes}

%    Information for first author
\author{D. Gonz\'alez Gonz\'alez}
%    Address of record for the research reported here
\address{Departamento de Matem{\'a}ticas, Universidad de Salamanca}
%    Current address
%\curraddr{Departamento de Matem{\'a}ticas, Universidad de Salamanca}
\email{davidgonzalezgonzalez1.7@usal.es}
%    \thanks will become a 1st page footnote.
\thanks{The first author is supported by a predoctoral contract funded by the Consejer\'ia de Educaci\'on of the Junta de Castilla y Le\'on and cofounded by the European Social Fund Plus (FSE$+$)}

%    Information for second author
\author{A. L. Mu{\~n}oz Casta{\~n}eda}
\address{Departamento de  Matem{\'a}ticas, Universidad de Le{\'o}n}
\email{amunc@unileon.es}
%\thanks{The second author was supported  by ......}

\author{L. M. Navas Vicente}
\address{Departamento de Matem{\'a}ticas, Universidad de Salamanca}
\email{navas@usal.es}

\thanks{The three authors are supported  by the project PID2023-150787NB-I00, funded by MICIU /AEI /10.13039/501100011033 /FEDER, UE}

%    General info

%\date{\today}

%\dedicatory{This paper is dedicated to ......}

\maketitle

%\leavevmode
%%

%%%%%%%%%%%%%%%%%%%%%%%%%%%%%%%%%%%%%
%%%%%%%%%%%%%%%%%%%%%%%%%%%%%%%%%%%%%
%%%%%%%%%%%%%%%%%%%%%%%%%%%%%%%%%%%%%
%%%%%%%%%%%%%%%%%%%%%%%%%%%%%%%%%%%%%

%\begin{description}
%
%\item[\sffamily Archivo fuente] {\sffamily\textcolor{Red3}{\currfilename}}
%
%\item[\sffamily Última modificación] {\sffamily\textcolor{Red3}{\today:\currenttime}.}
%
%\end{description}

\begin{abstract}
Rosenbloom and Tsfasman, in their foundational work on the $m$-metric, introduced algebraic-geometric codes defined by multiple points on a smooth projective curve $X$. This construction involves a divisor $G$ and another divisor $D=\sum n p_i$, where $p_i$ are distinct rational points with $p_i \notin \text{supp}(G)$ and $n\in\mathbb{N}$. Although these codes are significant, their formal development for arbitrary genus remains incomplete in the literature, as most studies have concentrated on the genus $0$ case.

We present a rigorous treatment of this class of codes. Starting with a smooth projective curve $X$, an invertible sheaf $L$, and an effective divisor $D=\sum n_i p_i$ where the $n_i$ are not necessarily equal, as well as tuples of uniformizers $t_D$ at the points of $D$ and trivializations $\gamma_D$ for the localizations $L_{p_i}$, the associated differential Goppa code is defined. This code arises from the theory of $n$-jets of invertible sheaves on curves, which enables the description of codewords using Hasse-Schmidt derivatives of sections of $L$.

The variation of the code under changes in the data $(t_D, \gamma_D)$ is examined, and the group acting on these parameters is described. The behavior of the minimum Hamming distance under such variations is analyzed, with explicit examples provided for curves of genus $0$ and $1$. A duality theorem is established, involving principal parts of meromorphic differential forms. It is demonstrated that Goppa codes constitute a proper subclass of differential Goppa codes, and that every linear code admits a differential Goppa code structure on $\mathbb P^1$ using only two rational points.
\end{abstract}

\tableofcontents{}

%%%%%%%%%%%%%%%%%%%%%%%%%%%%%%%%%%%%%%%%%%%
%%%%%%%%%%%%%%%%%%%%%%%%%%%%%%%%%%%%%%%%%%%

%-------------------------------------------------
\section{Introduction}

Classical algebraic-geometric codes over finite fields are constructed by evaluating global sections of invertible sheaves at a set of distinct rational points on a smooth projective curve. This construction gives rise to geometric Goppa codes and constitutes one of the most fruitful links between algebraic geometry and coding theory; see, for example, \cite{Goppa1983,TsfasmanBasic, Stichtenoth}. In this framework, each evaluation point has multiplicity one, so the encoded information is described only by the values of the sections at the fixed points.

However, from a geometric perspective, it is natural to consider effective divisors with arbitrary multiplicities. When a point appears with a multiplicity greater than one, the local structure of the curve suggests that the relevant information should no longer be limited to the value of a section, but should also incorporate its higher-order data. This leads immediately to the language of jets, sheaves of principal parts, and Hasse--Schmidt derivatives (see \cite{EGAIV,Laksov1984,LaksovThorup1991,Laksov1994,PerkinsonCurves,Piene1977}). From this viewpoint, the generalization of classical evaluation codes consists of replacing simple pointwise evaluation with an infinitesimal evaluation whose order is controlled by the multiplicities of the divisor.

The emergence of such ideas in coding theory can be traced back to the work of Rosenbloom and Tsfasman on the $m$-metric \cite{RT}, 
where the presence of ordered blocks reflects precisely the relevance of higher-order local information. 
Nevertheless, the explicit developments that have appeared in the literature have essentially concentrated on the genus-zero case; that is, on constructions over the projective line. To the authors’ knowledge, the only exception is \cite{CanMontero2025}, where the Bottleneck metric is studied over curves of arbitrary genus.
In that context, multiple evaluations recover and extend Reed--Solomon-type families through derivatives or hyperderivatives, giving rise to so-called multiplicity codes and affine variants \cite{Bhandari2024, CanHorowitz2025}.

The genus-zero case has demonstrated that these constructions possess both structural and algorithmic interest. 
In particular, they have yielded list decoding results with near-optimal parameters for univariate codes with multiplicities \cite{Kopparty2012, Guruswami2013, Goyal2024}; they have led to high-rate families with locality properties and sublinear decoding \cite{Kopparty2014, Wu2015}; and they have motivated efficient systematic encoding algorithms \cite{Coxon2019}. 
More recently, extensions based on hyperderivatives and divided differences have been studied, continuing to operate essentially within the univariate framework corresponding to $\mathbb{P}^1$ \cite{Bhandari2024, CanHorowitz2025}. 
This development highlights that considering multiplicities at evaluation points is not a minor refinement, but a source of new families of codes with distinctive properties.

The aim of this work is to introduce and study a geometric generalization of these constructions for smooth projective curves of arbitrary genus. 
Given a geometrically irreducible smooth projective curve $X$ defined over $\mathbb{F}_q$, an invertible sheaf $ L $ over $X$, and an effective divisor $D=\sum_{i=1}^s n_i p_i$.
We define a linear code associated with the evaluation of jets of global sections of $L$ along $D$. 
This construction depends explicitly on the choice of local uniformizers $t_D$ at the points in the support of $D$ and on suitable trivializations $\gamma_D$ of $L$ at those points. 
Concretely, the codewords of the code $C(X,L,D,t_D,\gamma_D)$ associated with these data are described by the Hasse--Schmidt derivatives of the sections with respect to these local parameters. When all multiplicities are equal to one, the classical situation of geometric Goppa codes is  recovered.
Moreover, while the minimum distance with respect to the Rosenbloom--Tsfasman metric and the rank metric is invariant under changes in the local data $(t_D,\gamma_D)$, the minimum Hamming distance generally depends on these choices. Understanding this dependence and how the local parameters may be used to control the resulting Hamming distance will be central to this work.

\subsection*{Main results and organization}
In Section 2, we introduce the notation and basic notions used throughout the article. In Section 3, we establish an intrinsically geometric formulation of differential Goppa codes over smooth projective curves of arbitrary genus (Definition~\ref{def:Goppa}), describe their codewords in terms of Hasse--Schmidt derivatives (Remark~\ref{rmk:code-words}), and determine their dimension (Lemma~\ref{lm:dimension Goppa}).  
We analyze the dependence of the construction on local uniformizers and trivializations, identifying the group (called the Taylor group) that governs these changes (Definition~\ref{def:taylor-group}). We describe the induced transformation rule on the coefficients of codewords (Theorem~\ref{th: taylor parameters}) and prove that Taylor-type trivializations form a single, strictly proper orbit inside the space of all possible trivializations (Proposition~\ref{prop:tay-simply-transitive} and Lemma~\ref{lm:taylor-orbit}).  
In Section 4, we present explicit examples in genus $0$ and $1$.  
In Section 5, we introduce the natural block-wise bilinear form used to define duality, construct the corresponding dual Taylor trivializations (Lemma~\ref{lm:map theta} and Lemma~\ref{prop: charac theta(gamma)}), establish the local residue--Hasse derivative convolution identity (Proposition~\ref{lem:hasse-residue-convolution}), and prove the duality theorem for differential Goppa codes (Theorem~\ref{th:duality}).  
In Section 6, we study the dependence of the minimum distance on the choice of local data. We introduce the intrinsic block distance (Definition~\ref{def:block-metric} and Definition~\ref{def:dblk-code}), compare it with the Hamming distance (Lemma~\ref{lm:blk-vs-hamming}), and show that the block distance is exactly the smallest Hamming distance that can be achieved by varying the Taylor trivializations (Theorem~\ref{th:dblk-achievable}). We then establish a bound for the size of the base field ensuring the existence of parameters with prescribed lower bounds on the Hamming distance (Theorem~\ref{th:large-q-existence}).  
Finally, in Section 7, we obtain two structural results. On the one hand, we show that every linear block code can be realized as a differential Goppa code over $\mathbb{P}^1$ (Theorem~\ref{th: every linear code}). On the other hand, we prove that strong geometric Goppa codes form a proper subclass of strong differential Goppa codes (Proposition~\ref{prop:numerical-obstruction-Fq} and Theorem~\ref{thm:proper-subclass-strong-Fq}). 

We also include an appendix that collects the basic notions and fundamental properties of jet bundles over algebraic curves used throughout the paper.

\subsection*{Remark on Preliminaries and Appendix}
%
%Although the general context is that of algebraic-geometric codes, the arguments developed in this work rely systematically on local and infinitesimal properties, like jet sheaves, truncated Taylor morphisms, and Hasse--Schmidt derivatives, that typically lie outside the standard treatment of the subject (see \cite{EGAIV,Laksov1984,LaksovThorup1991,Laksov1994,PerkinsonCurves,Piene1977}). Therefore, the preliminaries section provides a self-contained account of the geometric and differential notions used throughout the article. Furthermore, the appendix collects the essential results on jet bundles that are needed for defining differential Goppa codes, for the analysis of their dependence on local data, and for the formulation of the duality theorem. 
While the general context concerns algebraic-geometric codes, the arguments presented here systematically rely on local and infinitesimal properties, such as jet sheaves, truncated Taylor morphisms, and Hasse–Schmidt derivatives, which typically fall outside standard treatments of the subject (see \cite{EGAIV,Laksov1984,LaksovThorup1991,Laksov1994,PerkinsonCurves,Piene1977}). The preliminaries section offers a self-contained exposition of the geometric and differential concepts employed throughout the article. The appendix compiles the essential results on jet bundles required to define differential Goppa codes, examine their dependence on local data, and establish the duality theorem.

%-------------------------------------------------
\section{Background}

For the necessary background on algebraic geometry, we refer the reader to \cite{HartshorneAG, Stichtenoth}. Nevertheless, to make this work self-contained, we recall briefly the basic notions used throughout the article in this section.

%-------------------------------------------------
\subsection{Curves}

Throughout this article, $X$ denotes a smooth projective geometrically irreducible curve of genus $g$ over a finite field $\mathbb{F}_{q}$.  We denote by $O_X$ the structure sheaf of $ X$. If $x\in X$ is a point, we denote by $O_x$ the localization of $O_X$ at $x$ and by $\mathfrak{m}x$ the maximal ideal of $O_x$. We denote by $\mu$ its generic point and by $K_X$, or just $K$, its function field, $O{\mu}$. A rational function of $X$ is an element $f$ of $K$, and it defines a morphism of schemes $f:X\rightarrow \mathbb{P}^{1}$. Conversely, any morphism of schemes $f:X\rightarrow \mathbb{P}^{1}$ determines a rational function $f\in K$. Recall that given a rational function $f\in K$ and a closed point $x\in X$, either $f\in O_x$ or $f^{-1}\in O_x$. A rational function $f\in K$ is regular at $x\in X$ if $f\in O_x$.

Given a rational point $x\in X$ there is a canonical isomorphism $\gamma_x:O_x/\mathfrak{m}x\simeq \mathbb{F}{q}$. If $f\in K$ is a rational function regular at $x$, then the value of $f$ at $x$ is $f(x):=\gamma_x(f \pmod{\mathfrak{m}x})\in\mathbb{F}{q}$. Geometrically, $f(x)=[\gamma_x(f \pmod{\mathfrak{m}_x});1]$. We say that a closed point $x$ is a zero of $f$ of multiplicity $n\in\mathbb{N}$ if $f\in\mathfrak{m}_x^n\setminus \mathfrak{m}x^{n+1}$ and is a pole of $f$ of order $-n\in\mathbb{Z}{<0}$ if $f^{-1}\in \mathfrak{m}_x^n\setminus \mathfrak{m}_x^{n+1}$. The order of a zero or a pole of $f$ will be denoted by $\textrm{ord}_x(f)$.

\subsection{Differentials and residues}

We denote by $\omega_{X}$ the canonical sheaf  of X which we may identify with $\Delta^{}(I/I^2)$, $\Delta:X\hookrightarrow X\times X$ being the diagonal (closed) immersion and $I$ the ideal sheaf of $\Delta(X)$. This is an invertible sheaf. We denote by $\Omega_X^1$ the localization $\omega_{X,\mu}$. Elements of the one-dimensional $K$-vector space $\Omega_X^1$ are called meromorphic differentials of $X$.
We denote by $d:O_X\rightarrow \omega_X\subseteq \pi_{1}(O_X\otimes O_X) /I^2$ the canonical $\mathbb{F}{q}$-linear morphism of sheaves,
$$f\mapsto df= (f\otimes 1-1\otimes f) \pmod{I^2},$$ which is called the universal differential.
Here, $\pi{i}:X\times X\rightarrow X, \ i=1,2$, is the projection onto the $i$-th factor.
We will use the same notation $d$ for the localized map $d:O_x\rightarrow \omega_{X,x}$ at a point $x$.

A uniformizer of a closed point $x\in X$ is a rational function of $X$ regular at $x$, $t_x\in O_x\subset K$, such that $dt_x$ generates $\omega_{X,x}$. Given a uniformizer $t_x$ and a rational function $f\in K$, we denote by $f(t_x)$ the Laurent expansion of $f$ at $x$ with respect to $t_x$ (see \ref{eq:laurent expansion}). If $f$ is regular at $x$, we denote by $D_{t_x}^n(f)$ the  $n$-th derivative of $f$ at $x$ with respect to $t_x$ (see \ref{eq:derivative at a point}).

Given a meromorphic differential $\eta\in K$ and a closed point $x\in X$, we may choose a uniformizer $t_x$ and a representation $\eta=f dg$, $f$ and $g$ being rational functions. The residue of $\eta$ at $x$ is defined as (see \cite{Tate})
\begin{equation}\label{eq:residue}
\Res_{x}(\eta):=\left[\begin{array}{l}\textrm{degree $-1$ coefficient  in the Laurent}\\ \textrm{expansion of $f(t_x)\cdot g’(t_x)$} \end{array}\right]
\end{equation}
where $g’(t_x)$ stands for $D^1_{t_x}(g)$. The residue does not depend on the representation nor on the uniformizer, and it satisfies
\begin{equation}\label{eq: residue theorem}
\sum_{x\in X}\Res_{x}(\eta)=0, \ \forall \eta\in\Omega^1_X.
\end{equation}

\subsection{Divisors and invertible sheaves}
Given a divisor, $G=\sum_{x\in X}n_x x$, on $X$ we denote by $\textrm{mult}x(G)$ the multiplicity of $G$ at $x$; that is, $n_x$. The restriction of $G$ to an open subset $U\subset X$ is just $G_U:=\sum{x\in U}n_x x$.
Given a rational function $f\in X$, we denote by $(f)0$ and by $(f)\infty$ its divisors of zeroes and poles, respectively, and by $(f)=(f)0-(f)\infty$ its total divisor.
We denote by $O_X(G)$ its associated invertible sheaf. Recall that
\begin{equation}\label{eq:local sections}
O_X(G)(U):={f\in K| \ (f){U}+G{U}\geq 0}\cup{0}.
\end{equation}
If $L$ is an invertible sheaf on $X$, we denote $L\otimes O_X(G)$ by $L(G)$. Recall also that
\begin{equation}\label{eq:sec differentials}
\omega_X(G)(U)={\eta\in\Omega_X^1| \ (\eta)U+G_U\geq 0},
\end{equation}
Given a rational point $x\in X$, an nvertible sheaf $L$ on $X$ and a local section $s\in L(U)$ with $x\in U\subset X$, the composition of the restriction map $L(U)\rightarrow L_x/\mathfrak{m}x L_x$ with a trivialization $\phi: L_x/\mathfrak{m}x L_x\simeq \mathbb{F}{q}$ leads to a $\mathbb{F}{q}$-linear map
$$
ev_{\phi}:L(U)\rightarrow \mathbb{F}{q},
$$
and $ev\phi(s)$ is called the value of $s$ at $x$ with respect to $\phi$. If $\phi’$ is a different trivialization, then $ev_\phi’=\lambda \cdot ev_\phi$ for certain $\lambda\in\mathbb{F}_{q}^{\times}$.

\subsection{Sheaves of jets}

Given an invertible sheaf $L$, the sheaf of $i$-jets of $L$ is defined as
\begin{equation*}%\label{eq:n-jets}
J^i_{X}(L):=\pi_{1*}\left(\pi_{2}^{}L\otimes_{O_{X\times X}}O_{X\times X}/I^{i+1}\right).
\end{equation*}
Roughly speaking, $J^i_{X}(L)$ is the $O_X$-module whose fiber at a point $p\in X$ is
\begin{equation}\label{eq:fibers of jets}
J_p^{i}(L):=J^i_{X}(L)|_p\simeq L_p/\mathfrak{m}p^{i+1}L_p=L|{(i+1)p}.
\end{equation}
The sheaves $J^{i}X(L)$ come equipped with a canonical map $d_L^{i}:L\rightarrow J^i{X}(L)$. Once a local trivialization $\gamma:L|_U\simeq O_U$ has been fixed, and a generator $z$ of $\omega_X$ (shrinking $U$ if necessary) has been chosen,  $d_L^{i}$ sends a local section $s$ to its Taylor expansion of order $i$ with respect to $z$.

The construction of $J_X^{i}(L)$ is functorial, i.e., given a morphism of modules $\gamma:L\rightarrow L’$ there is an induced morphism
$$
J^{i}(\gamma):J^i_{X}(L)\rightarrow J^i_{X}(L’)
$$
which is compatible with the canonical maps $d_L^{i}$ and $d_{L’}^{i}$.

Jets of invertible sheaves play a crucial role in the article, and we have collected  the fundamental  properties in Appendix \ref{app: jets}.

\section{Differential Goppa codes I. Construction and properties}

For the rest of this section, we fix a smooth projective curve $X$ of genus $g$ over the finite field $\mathbb{F}_q$, $L$ an invertible sheaf and $D=n_1p_1+\cdots +n_sp_s$ an effective divisor of $X$.

\subsection{The Taylor series map at an effective divisor}\label{sec:pre-definition}

Let $n:=\textrm{max}_{i}\{n_i\}$ be the maximum multiplicity of $D$ and $\underline{n}$ the tuple $(n_1,\hdots,n_s)$. For each natural number $i$, we have the sheaf of jets of $L$ of degree $i$, $J^i_{X}(L)$. From the fundamental short exact sequence \ref{eq:fundamental exact sequence}, it follows the existence of a surjection for each $n_i$
$$
\psi_{i}:J^{n-1}_X(L)\rightarrow J^{n_i-1}_X(L)
$$
%whose kernel is $\textrm{Ker}(\psi_{i})=\pi_{1*}(\pi_2^* L\otimes_{O_{X\times X}} I^{n_i}/I^{n})$.
Moreover, composing with the evaluation map at a point $p_i\in X$, $J^{n_i-1}_X(L)\rightarrow J^{n_i-1}_p(L)$, we obtain a surjective morphism
$$
\psi_{i}(p):=J^{n-1}_X(L)\rightarrow J^{n_i-1}_p(L).
$$
Considering this and the initial set-up, we get a surjective morphism $\oplus_i \psi_{i}(p_i)$, which we denote by $\Psi_{D}$,
$$
\xymatrix{
\Psi_{D}:J^{n-1}_X(L)\rightarrow  \oplus_{i=1}^{s}J^{n_i-1}_{p_i}(L).
}
$$
%whose kernel is $\textrm{Ker}(\Psi_{D})=\cap_i  \textrm{Ker}(\psi_{i}(p_i))$.
On the other hand, we have the canonical map $d^{n-1}_L:L\rightarrow J^{n-1}_X(L)$ (Definition \ref{def:truncated Taylor map}). Taking global sections, pulling back to $X$ and composing with the evaluation map, we get a morphism of $O_X$-modules 
$$(ev \circ H^0(d^{n-1}_L)\otimes 1=) \ \partial^{n-1}_L:H^0(X,L)\otimes O_X\rightarrow J^{n-1}(L).$$ 
Thus, we get the diagram
$$
\xymatrix{
0\ar[r] & \Ker(\Psi_{D})\ar[r] &J^{n-1}_X(L)\ar[r]^-{\Psi_{D}} &   \oplus_{i=1}^{s}J^{n_i-1}_{p_i}(L) \ar[r] & 0\\
& & H^0(X,L)\otimes O_X \ar[u]^{\partial_{L}^{n-1}}   \ar@{-->}[ur]_{\Theta_{D}:=\Psi_{D}\circ \partial^{n-1}_L} &
}
$$
\begin{definition}\label{def: theta map}
The Taylor series map of $L$ at the effective divisor $D$ is defined as the morphism
$$
\Theta_D:H^0(X,L)\otimes O_X \rightarrow  \oplus_{i=1}^{s} J^{n_i-1}_{p_i}(L)
$$
\end{definition}

\begin{lemma}\label{lm: Taylor map at divisor}
Let $U=\Spec(A)\subset X$ be an affine open subset that trivializes $L$. Fix a trivialization $\gamma:L|_U\simeq O_U$ and a generator $z\in A$ of $\omega_X|_U$. Let $\xi:=\delta_U(z)\in J_U^{n-1}$ and let  $\overline{\xi^j}:=J^{n-1}(\gamma)^{-1}(\xi^j)\in J_X^{n-1}(L)(U)$. Then for an element
$\sum_{\ell=1}^r s_\ell\otimes a_\ell \in H^0(X,L)\otimes O_X(U), \ s_\ell\in H^0(X,L),\ a_\ell\in O_X(U)$,
we have
\begin{enumerate}
\item  
$\partial_L^{n-1}(U)\!\left(\sum_{\ell=1}^r s_\ell\otimes a_\ell\right) =\sum_{j=0}^{n-1}
\left(\sum_{\ell=1}^r a_\ell\,D^j_{U,\gamma,z}(s_\ell)\right)\overline{\xi^j}$.
\item 
$\Theta_{D}(U)\!\left(\sum_{\ell=1}^r s_\ell\otimes a_\ell\right) =
\bigoplus_{p_i\in U}
\sum_{j=0}^{n_i-1}
\left(\sum_{\ell=1}^r
a_\ell(p_i)\, D^j_{t_{i}}(\gamma(s_\ell))
\right)
\overline{\xi^j_{p_i}}$ where, for  $p_i\in U$, we denote by $\overline{\xi^j_{p_i}}\in J^{n_i-1}_{p_i}(L)$ the image of $\overline{\xi^j}$ under $\psi_{i}(p_i)$ and $t_{i}$ is the localization of $z$ at $p_i$.
\item $H^0(\Theta_{D})\!\left(s\right) = \bigoplus_{i=1}^s \sum_{j=0}^{n_i-1} D^j_{t_{p_{i}}}(\gamma_i(s))\overline{\xi^j_{p_i}}$, where $t_{i}$ is a uniformizer at $p_i$ and $\gamma_i$ is a trivialization of $L_{p_i}$.
\end{enumerate}
\end{lemma}
\begin{remark}
Given a local trivialization $\gamma: L_p\simeq O_p$ and a global section $s\in H^0(X,L)$, $\gamma(s)$ must be understood as the image of $s$ by the composition $H^0(X,L)\rightarrow L_p\simeq O_p$.
\end{remark}
\begin{proof}
1) By construction, for every open subset $V\subset X$ and every
element $\sum_{\ell=1}^r s_\ell\otimes a_\ell \in H^0(X,L)\otimes O_X(V)$,
one has
\begin{equation}\label{eq:OX-linearity}
\partial^{L}_{n-1}(V)\!\left(\sum_{\ell=1}^r s_\ell\otimes a_\ell\right)
=
\sum_{\ell=1}^r a_\ell \cdot \left(d^{L}_{n-1}(s_\ell)\right)|_V,
\end{equation}
where $d^{L}_{n-1}(s_\ell)\in \Gamma(V,J^{n-1}_X(L))$ denotes the restriction to $V$ of the section of $J^{n-1}_X(L)$ induced by $s_\ell\in H^0(X,L)$. We now apply \eqref{eq:OX-linearity} with $V=U$ and compute each $\left(d^{L}_{n-1}(s_\ell)\right)|_U$ in the basis $\{\xi^j\}_{j=0}^{n-1}$. By the commutative square \ref{eq: Derivatives of sections of invertible sheaves}, we have $J^{n-1}(\gamma)\bigl(d^{L}_{n-1}(s_\ell)|_U\bigr)= d^{n-1}_U(\gamma(s_\ell))\in J^{n-1}_U$. Writing $d^{n-1}_U(\gamma(s_\ell))$ in the basis $\{1,\xi,\dots,\xi^{n-1}\}$, the coefficients are precisely the higher derivations $D^j_{U,z}(\gamma(s_\ell))$ (see \ref{eq:higher derivative of function}), hence $d^{n-1}_U(\gamma(s_\ell)) =\sum_{j=0}^{n-1} D^j_{U,z}(\gamma(s_\ell))\,\xi^j.$ By Definition \ref{def: derivative of sections}, $D^j_{U,\gamma,z}(s_\ell):=D^j_{U,z}(\gamma(s_\ell))\in A$. Pulling back through $(J^{n-1}(\gamma))^{-1}$ yields
$$
d^{L}_{n-1}(s_\ell)|_U
=
\sum_{j=0}^{n-1} D^j_{U,\gamma,z}(s_\ell)\,\overline{\xi^j}.
$$
Finally, substituting this expression into \eqref{eq:OX-linearity} (with $V=U$),
and using the $O_U$-module structure on $J^{n-1}_X(L)|_U$, we obtain the desired formula.

2) It is clear from 1) and the definition of $\Theta_D$ that $\Theta_{D}(U)\!\left(\sum_{\ell=1}^r s_\ell\otimes a_\ell\right) =
\bigoplus_{p_i\in U}
\sum_{j=0}^{n_i-1}
\left(\sum_{\ell=1}^r
a_\ell(p_i)\,D^j_{U,\gamma,z}(s_\ell)(p_i)
\right)
\overline{\xi^j_{p_i}}$. The formula follows now from \ref{eq: coincidence derivatives}.

3) It follows from 2) that $H^0(\Theta_{D})\!\left(s\right) = \bigoplus_{i=1}^s \sum_{j=0}^{n_i-1} D^j_{U_i,\gamma_i,z_i}(s)(p_i)\overline{\xi^j_{p_i}}$, where each triple $(U_i,\gamma_i,z_i)$ satisfies the extra condition $p_i \in U_i$. The final formula follows again from \ref{eq: coincidence derivatives}.
\end{proof}

\subsection{Definition of differential Goppa code}\label{sec:definition}

We will give now the central notion of this article.

Let $X$ be a smooth projective curve, $L$ an invertible sheaf on $X$ and $D=n_1p_1+\cdots +n_s p_s$ and effective divisor. For each $i=1,\hdots,s$ we fix a uniformizer $t_{i}$ of $p_i$ and a trivialization $\gamma_i: L_{p_i}\simeq O_{p_i}$. We denote by $\overline{\gamma_i}$ the induced isomorphism $J^{n_i-1}_{p_i}(L)\simeq \mathbb{F}_{q}^{n_{i}}$ and by $\overline{\gamma}$ the isomorphism $\oplus_{i=1}^{s}\overline{\gamma_i}$. We will use the notation 
\begin{equation*}
\begin{split}
t_D&=(t_{p_1},\hdots, t_{p_n}), \\
\gamma_D&=(\gamma_1,\hdots,\gamma_n),
\end{split}
\end{equation*}
in the rest of the article. Likewise, we denote by $ev_{t_D,\gamma_D}$ the composite linear map
$$
ev_{t_D,\gamma_D}: H^0(X,L)\overset{H^0(\Theta_D)}{\rightarrow}  \oplus_{i=1}^{s} J^{n_i-1}_{p_i}(L) \overset{\overline{\gamma}}{\simeq} \mathbb{F}_{q}^{M}.
$$

\begin{definition}\label{def:Goppa}
Let $\Gamma\subset H^0(X,L)$ be a linear subspace. 
We define the {differential Goppa code} associated to the data $X,L,D,\Gamma,t_{D},\gamma_{D}$ as the linear code of  length $M:=\sum_{i=1}^s n_i$ given by
$$
C(X,L,D,\Gamma,t_{D},\gamma_{D}):=\textrm{Im}\left( ev_{t_D,\gamma_D}|_{\Gamma} \right)
$$
We say that $C(X,L,D,\Gamma,t_{D},\gamma_{D})$ is a {uniform differential Goppa code} if 
$$n_1=\cdots =n_s$$ 
in the divisor $D$. In case $\Gamma$ is the complete linear series, 
$$\Gamma=H^0(X,L),$$ 
we write $C(X,L,D,t_{D},\gamma_{D})$ instead of $C(X,L,D,H^0(X,L),t_{D},\gamma_{D})$.
\end{definition}

\begin{remark}\label{rmk:code-words}
\begin{enumerate}
\item It follows from Lemma \ref{lm: Taylor map at divisor}, 3) that any word of $C(X,L,D,\Gamma,t_{D},\gamma_{D})$ is of the form
\begin{equation}\label{eq:Goppa word}
\begin{split}
ev_{t_D,\gamma_D}(s):=\Big(D^0_{t_{p_{1}}}(\gamma_1(s)), D^1_{t_{p_{1}}}(\gamma_1(s)),&\hdots, D^{n_1 -1}_{t_{p_{1}}}(\gamma_1(s)),\\
&\vdots\\
D^0_{t_{p_{s}}}(\gamma_s(s)), D^1_{t_{p_{s}}}(\gamma_s(s)),&\hdots, D^{n_s -1}_{t_{p_{s}}}(\gamma_s(s))\Big)
\end{split}
\end{equation}
for $s\in \Gamma$.
Recall that $D^0_{t_{p_{i}}}(\gamma_i(s))=\gamma_i(s)(p_i)$.
\item 
If $n_1=\hdots, n_s=1$, the code $C(X,L,D,\Gamma,t_{D},\gamma_{D})$ depends only on the classes $ \overline{\gamma_i}:=\gamma_i \mod{\mathfrak{m}_i}$ and does not depend on the uniformizers. Thus, $C(X,L,D,\Gamma,t_{D},\gamma_{D})=C(X,L,D,\Gamma,\overline{\gamma_D})$ is a geometric Goppa code in the sense of Tsfasman--Vl\u{a}du\c{t} \cite[\S 3.1.1]{TsfasmanBasic}. If furthermore, 
$L=O_X(G)$ with $\textrm{supp}(G)\cap \textrm{supp}(D)=\emptyset$, 
then $L_{p_i}=O_{p_{i}}$ and we may choose $id_D:=(id_{O_{p_1}},\hdots, id_{O_{p_s}})$ as $\gamma_D$. Indeed, $C(X,L,D,\Gamma,\overline{id_D})=C(X,G,D,\Gamma)$, is a geometric Goppa code (in the classical sense \cite{Goppa1983}) defined by the linear series $\Gamma$.
\end{enumerate}
\end{remark}

\begin{lemma}\label{lm:dimension Goppa}
In the above situation, it holds 
$$\Ker(H^0(\Theta_{D})) = H^0(X, L(-D)).$$
Therefore, 
$$
\Dim  \ C(X,L,D,t_{D},\gamma_{D})= \Dim  \ H^0(X,L)-\Dim  \ H^0(X,L(-D)).
$$
In particular, if $\sum_{i=1}^{s}n_i > \deg (L)$, then 
$$\Dim  \ C(X,L,D,t_{D},\gamma_{D})= \Dim  \ H^0(X,L).$$
\end{lemma}
\begin{proof}
By Lemma~\ref{lm: Taylor map at divisor}(2), and since the elements $\overline{\xi^j_{p_i}}$ are linearly independent, we have
$H^0(\Theta_D)(s)=0$ if and only if $D^j_{t_{i}}\big(\gamma_i(s)\big)=0$ for all  $i=1,\dots,s \text{ and } j=0,\dots,n_i-1$.
Fix an index $i$. Because $s$ is regular at $p_i$, its local expression in the completed local ring
$\widehat{O}_{X,p_i}\simeq k[[t_{i}]]$ can be written as
\[
\gamma_i(s)(t_i)=\sum_{m\ge0} a_m t_{i}^m.
\]
By definition of the differential operators $D^j_{t_{i}}$, we have
$D^j_{t_{i}}(\gamma_i(s))=a_j$. Therefore, $D^j_{t_{i}}(\gamma_i(s))=0 \text{ for all } j<n_i$ if and only if $\gamma_i(s)\in (t_{i}^{n_i})=\mathfrak m_{p_i}^{n_i}$.
Undoing the trivialization, this condition is equivalent to $s\in \mathfrak m_{p_i}^{n_i}L_{p_i}$.
Combining everything together, we obtain
\begin{equation}\label{eq: kernel Goppa}
s\in\Ker\big(H^0(\Theta_D)\big)
\quad\Longleftrightarrow\quad
s\in \mathfrak m_{p_i}^{n_i}L_{p_i} \ \text{for all } i=1,\dots,s.
\end{equation}
On the other hand,  a section $s\in H^0(X,L)$ lies in $H^0(X,L(-D))$ if and only if
$s_{p_i}\in \mathfrak m_{p_i}^{n_i}L_{p_i}$ for all $i$.
Comparing with \ref{eq: kernel Goppa}, we conclude that
\[
\Ker\big(H^0(\Theta_D)\big)=H^0\big(X,L(-D)\big),
\]
as claimed. The second part is obvious.
\end{proof}

\begin{remark}\label{rmk: RT}
It follows from Definition \ref{def:Goppa}, Lemma \ref{lm: Taylor map at divisor} and Lemma \ref{lm:dimension Goppa}, that the ``evaluation map'' defining a differential Goppa code is just the linear map $H^0(X,L)\rightarrow L|_{D}$ obtained by taking global sections on the canonical surjection $L\rightarrow L|_D$. Therefore, uniform differential Goppa codes are the algebraic geometric codes for the $m$-metric defined by Rosenbloom and Tsfasman in \cite[\S 3]{RT}.
\end{remark}

\subsection{Dependency on uniformizers and trivializations}

Our aim now is to study how the code $C(X,L,D,\Gamma,t_{D},\gamma_{D})$ changes when $t_{D}$ or $\gamma_{D}$ vary while keeping $X,L,D$ unchanged. The product formula for derivations and the Faa di Bruno formula (see \ref{sec: variations}) play an important role here.

\begin{lemma}\label{lm:change of uniformizer at a point}
Let $X$ and $D$ be as above, $p\in X$ a rational point, $\gamma:L_p\simeq O_p$ a trivialization and $t$ and $u$ uniformizeres at $p$. Define $T_p\in\mathrm{GL}_n(\mathbb{F}_{q})$ as the lower triangular matrix with entries
\[
(T_p)_{m,j}:=
\begin{cases}
\sum_{\substack{(n_1,\dots,n_m)\in\mathbb N^m\\
\sum_{r=1}^m r\,n_r=m\\
\sum_{r=1}^m n_r=j}}
\binom{j}{n_1,n_2,\dots,n_m}\ \prod_{r=1}^m D^{r}_{u}(t)^{n_r}, & 0\le j\le m\le n-1,\\
0, & j>m.
\end{cases}
\]
Then, for any global section $s\in H^0(X,L)$, one has
\[
(D^0_u(\gamma(s)),\dots,D^{n-1}_u(\gamma(s)))
=
(D^0_t(\gamma(s)),\dots,D^{n-1}_t(\gamma(s)))\,T_p.
\]
\end{lemma}

\begin{proof}
Since $\gamma:L_p\simeq  O_p$ is an isomorphism of $O_{p}$--modules,
we may choose an affine neighbourhood $U=\Spec(A)\subset X$ of $p$
trivializing $L$, together with a trivialization
$\gamma_U:L|_U\simeq  O_U$ whose localization at $p$, $(\gamma_U)_p$, coincides
with $\gamma$.
Shrinking $U$ further if necessary, we may also choose generators $z,w\in A$ of
$\omega_X|_U$ such that
\[
t=\frac{z}{1}\in O_{p},
\qquad
u=\frac{w}{1}\in O_{p}
\]
are the given uniformizers.
Then for every $j\ge 0$ and every $s\in H^0(X,L)$, $D^j_t(\gamma(s)) \;=\; D^j_{U,\gamma_U,z}(s)(p)$ and $D^j_u(\gamma(s)) \;=\; D^j_{U,\gamma_U,w}(s)(p)$, by \eqref{eq: coincidence derivatives}.
Now, applying the Fa\`a di Bruno formula \ref{eq: faa di bruno} on $U$ (for the change
of parameter from $z$ to $w$), evaluating at $p$ and using again
\eqref{eq: coincidence derivatives}, we obtain
\begin{equation*}
D^{m}_u(\gamma(s))
=
\sum_{j=0}^{m}
D^{j}_t(\gamma(s))
\left(
\sum_{\substack{(n_1,\dots,n_m)\in\mathbb{N}^m\\
\sum_{r=1}^m r n_r = m\\
\sum_{r=1}^m n_r = j}}
\binom{j}{n_1,\dots,n_m}
\prod_{r=1}^m \big(D^{r}_{u}(t)\big)^{n_r}
\right),
\end{equation*}
 for each $m=0,\dots,n-1$.
By definition of the matrix $T_p$, the inner parenthesis is $(T_p)_{m,j}$, hence the
above equality is equivalent to the claimed matrix relation.
\end{proof}

\begin{lemma}\label{lm:change of trivialization at a point}
Let $X$ and $D$ be as above, let $p\in X$ be a rational point, and let
$t$ be a uniformizer at $p$. Let $\gamma,\gamma':L_p\simeq O_p$ be two
trivializations, and let $f\in O_{p}^\times$ be such that $\gamma'=f\,\gamma$.
Fix $n\ge 1$ and define $S_p\in\mathrm{GL}_n(\mathbb{F}_{q})$ as the lower triangular matrix with
entries
\[
(S_p)_{m,j}:=
\begin{cases}
D^{m-j}_t(f), & 0\le j\le m\le n-1,\\
0, & j>m.
\end{cases}
\]
Then, for any global section $s\in H^0(X,L)$, one has
\[
\big(D_t^0(\gamma'(s)),\dots,D_t^{n-1}(\gamma'(s))\big)
=
\big(D_t^0(\gamma(s)),\dots,D_t^{n-1}(\gamma(s))\big)\,S_p.
\]
\end{lemma}

\begin{proof}
Since $\gamma,\gamma':L_p\simeq  O_p$ are isomorphisms of $O_{p}$--modules,
we may choose an affine neighbourhood $U=\Spec(A)\subset X$ of $p$
trivializing $L$, together with trivializations (still denoted)
$\gamma, \gamma':L|_U\simeq O_U$ whose localization at $p$ are the given maps
$\gamma,\gamma':L_p\simeq O_p$. 
Then $\gamma'=f\,\gamma$ for a unique $f\in O_{p}^\times$, and after shrinking
$U$ once more we may represent it by an element $f\in A^\times$ whose
localization in $O_{p}$ is the given unit relating $\gamma'$ and $\gamma$.

Let $z\in A$ be such that $dz$ generates $\omega_X|_U$ and whose localization $t=z/1$ is the chosen
uniformizer at $p$. By \eqref{eq: localization of derivation} and
\eqref{eq: coincidence derivatives}, for every $m\ge 0$ and every $s\in H^0(X,L)$ we have
\begin{equation}\label{eq:identify-triv}
D_t^m(\gamma(s))=D^m_{U,\gamma,z}(s)(p),\qquad
D_t^m(\gamma'(s))=D^m_{U,\gamma',z}(s)(p).
\end{equation}
Applying now \eqref{eq: trivialization derivation} to $\gamma' = f\gamma$ and then \eqref{eq:identify-triv}  yields
\begin{equation}\label{eq:pointwise-change-triv}
D_t^m(\gamma'(s))=\sum_{j=0}^{m} D^{m-j}_{U,z}(f)(p)\,D_t^{j}(\gamma(s)).
\end{equation}
Finally, by the same identification \eqref{eq: coincidence derivatives} (with $f$ in
place of $\gamma(s)$), we have
$D^{r}_{U,z}(f)(p)=D_t^{r}(f)$ for $r\ge 0$.
Substituting into \eqref{eq:pointwise-change-triv} gives, for $m=0,\dots,n-1$, $D_t^m(\gamma'(s))=\sum_{j=0}^{m} D_t^{m-j}(f)\,D_t^{j}(\gamma(s))$.
By definition of $S_p$, the previous identity is equivalent to the stated matrix relation.
\end{proof}

\begin{theorem}\label{th: taylor parameters}
Let $X$ and $D$ be as above, and $p\in X$ a rational point. Let $t,t'$ be uniformizers of $p$ and  $\gamma,\gamma':L_p\simeq O_p$ two
trivializations,. Let $f\in O_{p}^\times$ be such that $\gamma'=f\,\gamma$.
Fix $n\ge 1$. For any global section $s\in H^0(X,L)$, one has
\[
\big(D_{t'}^0(\gamma'(s)),\dots,D_{t'}^{n-1}(\gamma'(s))\big)
=
\big(D_t^0(\gamma(s)),\dots,D_t^{n-1}(\gamma(s))\big)\,S_p T_p.
\]
\end{theorem}
\begin{proof}
It follows directly by applying first Lemma \ref{lm:change of trivialization at a point} and then Lemma \ref{lm:change of uniformizer at a point}. It only remains to show why one does not obtain $T_pS_p$ by reversing the steps.

If one first changes the uniformizer while keeping $\gamma$ fixed, Lemma \ref{lm:change of uniformizer at a point} gives
\[
\big(D_{t'}^0(\gamma(s)),\dots,D_{t'}^{n-1}(\gamma(s))\big)=\big(D_{t}^0(\gamma(s)),\dots,D_{t}^{n-1}(\gamma(s))\big)T_p.
\]
To pass from $\gamma$ to $\gamma'=f\gamma$ at the parameter $t'$, one must then apply
Lemma~\ref{lm:change of trivialization at a point} with the uniformizer $t'$,
which gives a matrix
\[
S_p':=\big(D_{t'}^{\,m-j}(f)\big)_{0\le j\le m\le n-1},
\]
not the matrix $S_p$ (whose entries are $D_t^{m-j}(f)$). Thus the reversed procedure
yields 
$$
\big(D_t'^0(\gamma'(s)),\dots,D_t'^{n-1}(\gamma'(s))\big)
=
\big(D_t^0(\gamma(s)),\dots,D_t^{n-1}(\gamma(s))\big) T_p\,S_p'.
$$
The two matrices $S_p$ and $S_p'$ are related by conjugation:
indeed, applying Lemma \ref{lm:change of uniformizer at a point} to the function $f$ itself gives
\[
\big(D_{t'}^0(f),\dots,D_{t'}^{n-1}(f)\big)=\big(D_t^0(f),\dots,D_t^{n-1}(f)\big)\,T_p,
\]
and this is equivalent to
\[
S_p' = T_p^{-1}\,S_p\,T_p.
\]
Consequently, $T_p\,S_p' = T_p\,(T_p^{-1}S_pT_p)=S_pT_p$, so reversing the steps gives the same final matrix $S_pT_p$.
\end{proof}

\subsection{The Taylor group}

Motivated by the above lemmas, we define now the abstract group  acting on the data $t_{D},\gamma_{D}$ defining a differential Goppa code. 
The concepts and results stated here will appear again further on.

Let $X$ be a smooth curve over a finite field $\mathbb{F}_{q}$, let $p\in X$ be a rational point,
and let $L$ be an invertible sheaf on $X$. Fix an integer $n\ge 1$ and recall that
$J^n_p =  O_{p}/\mathfrak m_p^{\,n+1}$ and $J^n_p(L) = L_p/\mathfrak m_p^{\,n+1}L_p$.

\begin{definition}\label{def:taylor-group}
The {Taylor group of order $n$ at $p$} is the semidirect product
$$
G_n(p) := (J^n_p)^\times \rtimes \Aut_{\mathbb{F}_{q}-alg}(J^n_p),
$$
where the group law $\varphi:G_n(p) \times G_n(p)  \rightarrow G_n(p)$ is given by
\begin{equation*}
(a_1,\sigma_1)\cdot (a_2,\sigma_2)  \mapsto (a_1\cdot \sigma_1(a_2), \sigma_1\circ\sigma_2)
\end{equation*}
\end{definition}
\begin{remark}\label{rmk: Taylor algebraic group}
Note that $G_n(p)$ is an affine algebraic group over $\mathbb{F}_q$. 
%Its coordinate ring is generated by the coefficients of the invertible jets and the coefficients of the automorphisms, with group operations (multiplication and inversion) given by regular maps.
\end{remark}
\begin{definition}
The {fundamental action of the Taylor group $G_n(p)$} is the action on $J^n_p$ defined by the rule
\begin{equation}\label{eq:fundamental action}
\begin{split}
\Xi:G_n(p) \times J^n_p & \rightarrow J^n_p\\
((a,\sigma), r) & \mapsto a\cdot \sigma(r)
\end{split}
\end{equation}
\end{definition}

\begin{lemma}\label{lm:affine-decomposition}
The fundamental action is faithful. 
\end{lemma}
\begin{proof}
We have to show that  if $a_1\sigma_1(r)=a_2\sigma_2(r)$ for all $r\in J$,
with $a_i\in J^\times$ and $\sigma_i\in \Aut_{\F_q\text{-alg}}(J)$, then $a_1=a_2$ and $\sigma_1=\sigma_2$.
Since $\sigma_i(1)=1$,  evaluating at $r=1$, we get $a_1=a_2$.
Multiplying by $a_1^{-1}$ yields $\sigma_1(r)=\sigma_2(r)$ for all $r$, hence $\sigma_1=\sigma_2$.
\end{proof}

\subsubsection{The fundamental representation}

Fix a uniformizer $t_p$ at $p$. Using the basis
$\{1, t_p ,\dots, t_p ^{n-1}\}$ of $J^n_p$, we obtain an isomorphism of $\mathbb{F}_{q}$-vector spaces $u_{t_p}: J^n_p \xrightarrow{\ \sim\ } \mathbb{F}_{q}^n$.
\begin{definition}
The {fundamental linear representation of the Taylor group $G_n(p)$} is the linear representation
\[
\rho_{t_p}: G_n(p) \longrightarrow \GL_n(\mathbb{F}_{q}).
\]
induced from $\Xi$ and $u_{t_p}: J^n_p \xrightarrow{\ \sim\ } \mathbb{F}_{q}^n$.
\end{definition}
Our aim now is to describe $\rho_{t_p}$ explicitly. To this end, we fist establish some notation. Let $(a,\sigma)\in G_n(p)$, where $a\in (J^n_p)^\times$ and
$\sigma\in\Aut_{\mathbb{F}_{q}-alg}(J^n_p)$. Let $f\in O_p^{\times}$ be such that 
\begin{equation}\label{eq: representation-1}
a= \overline{f}:=f \pmod{\mathfrak{m}_p^n}.
\end{equation}
Then, $a\mapsto f(t_p) \pmod{t_p^n}$ gives an isomorphism of $\mathbb{F}_{q}$-algebras $J^n_p\simeq \mathbb{F}_{q}[t_p]/t_p^n$.
Under this isomorphism, $\sigma( t_p )$ can be written as
\begin{equation}\label{eq: representation-2}
\sigma( t_p )=\sum_{r=1}^{n-1} c_r t_p^r, \quad c_1\neq 0,
\end{equation}
and this expression determines completely $\sigma$ since $ t_p $ generates $J_p^n$.

\begin{lemma}\label{lm:action local parameters}
The matrix $\rho_{t_p}(a,\sigma)$ is the product
\[
\rho_{t_p}(a,\sigma) := S_p(a) \, T_p(\sigma),
\]
where $S_p(a)$ and $T_p(\sigma)$ are the lower triangular matrices defined by
\begin{equation*}
\begin{split}
(S_p(a))_{m,j} &= D_{t_p}^{m-j}(f), \\
(T_p(\sigma))_{m,j} &= \sum_{\substack{(k_1,\dots,k_m)\in\mathbb N^m \\ k_1 + \dots + k_{m-j+1} = j \\ \sum_{i=1}^{m-j+1} i k_i = m}} \binom{j}{k_1, \dots, k_{m-j+1}} \prod_{i=1}^{m-j+1} c_i^{k_i},
\end{split}
\end{equation*}
with $0\le j\le m\le n-1$.
\end{lemma}
\begin{remark}
These are precisely the matrices describing the change of trivialization and the change
of uniformizer in the definition of differential Goppa code.
\end{remark}
\begin{proof}
Consider the basis $\{1,  t_p , \dots, t_p^{n-1}\}$ of the $\mathbb{F}_{q}$-vector space $J^n_p$. 
Write $r = \sum_{j=0}^{n-1} r_j t_p^j$. By the definition of the linear representation $\rho_{t_p}$, \ref{eq: representation-1} and \ref{eq: representation-2}, we have:
\[
\rho_{t_p}(a, \sigma)(r) = a \cdot \sigma(r) = f(t_p) \cdot \sigma \left( \sum_{j=0}^{n-1} r_j t_p^j \right) \pmod{t_p^n}.
\]
Since $\sigma$ is a $\mathbb{F}_{q}$-algebra homomorphism, we have
\[
\rho_{t_p}(a, \sigma)(r) = f(t_p) \cdot \sum_{j=0}^{n-1} r_j (\sigma(t_p))^j \pmod{t_p^n}.
\]
Set $h = \rho_{t_p}(a, \sigma)(r) = \sum_{m=0}^{n-1} h_m t_p^m$. The matrix $\rho_{t_p}$ must satisfy the relation $h_m = \sum_j (\rho_{t_p})_{m,j} r_j$. Expanding the previous expression and using \ref{eq:laurent expansion} and \ref{eq:derivative at a point}, we obtain
\[
\sum_{m=0}^{n-1} h_m t_p^m = \left( \sum_{l=0}^{n-1} D^l_{t_p}(f) t_p^l \right) \sum_{j=0}^{n-1} r_j \left( \sum_{r=1}^{n-1} c_r t_p^r \right)^j \pmod{t_p^n}.
\]
Let $T_{s,j}$ denote the $s$-th coefficient  of $\sigma(t_p)^j$. Using the multinomial expansion formula derived for the $j$-th power of a polynomial we get
\[
T_{s,j} = \sum_{\substack{(k_1,\dots,k_m)\in\mathbb N^m \\ k_1 + \dots + k_{s-j+1} = j \\ \sum_{i=1}^{s-j+1} i k_i = s}} \binom{j}{k_1, \dots, k_{s-j+1}} \prod_{i=1}^{s-j+1} c_i^{k_i}.
\]
Substituting this into the sum and rearranging the order of summation, we get
\begin{equation*}
\begin{split}
\sum_{m=0}^{n-1} h_m t_p^m  &= 
\sum_{l=0}^{n-1} D^l_{t_p}(f) t_p^l \sum_{j=0}^{n-1} r_j \left( \sum_{s=j}^{n-1} T_{s,j} t_p^s \right) = \\
&=\sum_{j=0}^{n-1} \left( \sum_{l=0}^{n-1} \sum_{s=j}^{n-1} D^l_{t_p}(f) T_{s,j} t_p^{l+s} \right) r_j.
\end{split}
\end{equation*}
Taking terms of degree $m$ (where $m = l + s$), we find
\[
h_m = \sum_{j=0}^m \left( \sum_{s=j}^m D^{m-s}_{t_p}(f) T_{s,j} \right) r_j.
\]
This sum corresponds exactly to the definition of matrix multiplication for lower triangular matrices. In particular, 
$$(\rho_{t_p})_{m,j} = \sum_{s=j}^m S_{m,s} T_{s,j},$$ 
where $S_{m,s} =D^{m-s}_{t_p}(f)$ is the $(m,s)$-entry of the lower triangular matrix $S_p(a)$ and $T_{s,j}$ is the $(s,j)$-entry of the matrix representing the automorphism $\sigma$.
Therefore, $\rho_{t_p}(a, \sigma) = S_p(a) T_p(\sigma)$.
\end{proof}

\begin{remark}
It is easy to see that the fundamental representation shows that the Taylor group provides a geometric realization
of a triangular subgroup of the group of linear automorphism preserving the
Rosenbloom--Tsfasman metric, studied in depth in \cite{Lee2003}. That is, the Taylor group may be viewed as the
subgroup of Rosenbloom--Tsfasman type automorphisms arising from formal local
geometry, namely from changes of local parameter and changes of local trivialization.
\end{remark}

\subsubsection{The action on the space of trivializations}
Fix a trivialization $\gamma_p:L_p\simeq O_p$ and a uniformizer $t_p$ at $p$. Using the induced isomorphism $\overline{\gamma_p}:J^n_p(L)\simeq J^n_p$ and the basis
$\{1, t_p ,\dots, t_p ^{n-1}\}$ of $J^n_p$, we obtain an isomorphism of $\mathbb{F}_{q}$-vector spaces 
\begin{equation}\label{eq:iso-params}
u_{t_p,\gamma_p}: J^n_p(L) \xrightarrow{\ \sim\ } \mathbb{F}_{q}^n.
\end{equation} 

Let $\Iso_{\mathbb{F}_{q}}(J^n_p(L),\mathbb{F}_{q}^n)$ denote the set of $\mathbb{F}_{q}$--linear isomorphisms $u:J^n_p(L)\to\mathbb{F}_{q}^n$.  Let $\Iso_{\mathrm{Tay}}(J^n_p(L),\mathbb{F}_{q}^n)\subset\Iso_{\mathbb{F}_{q}}(J^n_p(L),\mathbb{F}_{q}^n)$ be the subset consisting
of those isomorphisms $u:J^n_p(L)\to\mathbb{F}_{q}^n$ for which there exist a uniformizer $t$ at $p$ and a
trivialization $\gamma:L_p\simeq  O_{p}$ such that $u=u_{t,\gamma}$.

Fix a trivialization $\gamma_p:L_p\simeq O_p$. Using the induced isomorphism $\overline{\gamma_p}:J^n_p(L)\simeq J^n_p$ we get an action
$$
\Psi:G_n(p)\rightarrow \textrm{Aut}(J^n_p(L))
$$  
as follows:
$g=(a,\sigma)\mapsto \Psi_g=(\overline{\gamma_p})^{-1}\circ \Xi_g \circ \overline{\gamma_p}$
where $\Xi_g(r)=a\cdot \sigma(r)$ (see \ref{eq:fundamental action}). The action of $G_n(p)$ on the set of trivializations $\Iso_{\mathbb{F}_{q}}(J^n_p(L), \mathbb{F}_{q}^n)$ is therefore given by 
\begin{equation}\label{eq:action on Taylor}
\begin{split}
\widehat{\Psi}:G_n(p) \times \Iso_{\mathbb{F}_{q}}(J^n_p(L),\mathbb{F}_{q}^n) & \rightarrow \Iso_{\mathbb{F}_{q}}(J^n_p(L),\mathbb{F}_{q}^n)\\
(g,u) & \mapsto (g \cdot u) := u \circ \Psi_g^{-1}
\end{split}
\end{equation}

\begin{lemma}
Let $u\in \Iso_{\mathrm{Tay}}(J_p^n(L),\F_q^n)$ and $g\in G_n(p)$. Then
\[
g\cdot u \;:=\; u\circ \Psi_g^{-1}\ \in\ \Iso_{\mathrm{Tay}}(J_p^n(L),\F_q^n).
\]
Equivalently, $\Iso_{\mathrm{Tay}}(J_p^n(L),\F_q^n)$ is stable under the action $\widehat\Psi$.
\end{lemma}

\begin{proof}
Choose a uniformizer $t$ at $p$ and a trivialization $\gamma:L_p\simeq O_p$ such that
$u=u_{t,\gamma}$ (cf.\ \eqref{eq:iso-params}). Let $g=(a,\sigma)\in G_n(p)$.
Recall that $\overline{\gamma}:J_p^n(L)\xrightarrow{\sim}J_p^n$ is an isomorphism of $J_p^n$--modules and that $\Psi_g=(\overline{\gamma})^{-1}\circ \Xi_g\circ \overline{\gamma}$,
$\Xi_g(r)=a\cdot \sigma(r)$. Hence
$g\cdot u=u_{t,\gamma}\circ \Psi_g^{-1}=u_{t,\gamma}\circ (\overline{\gamma})^{-1}\circ \Xi_g^{-1}\circ \overline{\gamma}$.
Since $u_{t,\gamma}=u_t\circ \overline{\gamma}$, where $u_t:J_p^n\xrightarrow{\sim}\F_q^n$ is the coefficient map in the basis
$\{1,t,\dots,t^n\}$, we get
\begin{equation}\label{eq:tay-closure-2}
g\cdot u
=
(u_t\circ \Xi_g^{-1})\circ \overline{\gamma}.
\end{equation}
%Thus it suffices to show that the $\F_q$--linear isomorphism $u_t\circ \Xi_g^{-1}:J_p^n\to \F_q^n$
%is itself of the form $u_{t'}\circ \alpha$ for some uniformizer $t'$ and some $J_p^n$--module automorphism
%$\alpha$ induced by a change of trivialization; then $g\cdot u$ will be of the form $u_{t',\gamma'}$.
Now we proceed step by step.
\begin{enumerate}
\item Since $\Xi_g(r)=a\cdot \sigma(r)$, we have 
\begin{equation}\label{eq:Xi-inverse-affine}
\Xi_g^{-1}(r)=\sigma^{-1}(a^{-1}r),
\end{equation}
for every $(r\in J_p^n$.
%Define the $\F_q$--algebra automorphism $\alpha:=\sigma^{-1}\in \Aut_{\F_q\text{-alg}}(J_p^n)$
%and the unit $b:=\sigma^{-1}(a^{-1})\in (J_p^n)^\times$.
%Then
%\begin{equation}\label{eq:Xi-inverse-affine}
%\Xi_g^{-1}(r)=b\cdot \alpha(r)
%\end{equation}

\item Let $\fm:=\fm_p/\fm_p^{n+1}$ be the maximal ideal of the local ring $J_p^n$.
Since $\sigma\in \Aut_{\F_q\text{-alg}}(J_p^n)$ is a ring automorphism, it preserves $\fm$ (the unique maximal ideal),
hence also $\fm^2$. Therefore $\sigma$ induces an $\F_q$--linear automorphism of the $1$--dimensional vector space
$\fm/\fm^2$. In particular, if $\overline t$ denotes the class of a uniformizer $t$ in $J_p^n$, then
$\sigma(\overline t)\in \fm\setminus \fm^2$.
Choose $t'\in \fm_p$ whose class in $J_p^n$ equals $\sigma(\overline t)$.
Then $t'\notin \fm_p^2$, hence $t'$ is a uniformizer at $p$. By construction,
\begin{equation}\label{eq:tprime-def}
\overline{t'}=\sigma(\overline t)\ \in J_p^n.
\end{equation}

\item Consider the $\F_q$--algebra isomorphisms
\begin{equation*}
\begin{split}
\theta_t:\F_q[T]/(T^{n+1})&\xrightarrow{\sim}J_p^n,\quad T\mapsto \overline t,\\
\theta_{t'}:\F_q[T]/(T^{n+1})&\xrightarrow{\sim}J_p^n,\quad T\mapsto \overline{t'}.
\end{split}
\end{equation*}
By definition, $u_t$ (resp.\ $u_{t'}$) is the composite of $\theta_t^{-1}$ (resp.\ $\theta_{t'}^{-1}$)
with the identification $\F_q[T]/(T^{n+1})\simeq \F_q^{n+1}$ given by the basis $\{1,T,\dots,T^n\}$.
Equation \eqref{eq:tprime-def} implies $\sigma^{-1}\circ \theta_t'=\theta_{t}$.
Therefore,
\begin{equation}\label{eq:coeff-change}
u_t\circ \sigma^{-1}=u_{t'}.
\end{equation}

\item By \eqref{eq:Xi-inverse-affine} and \eqref{eq:coeff-change}, we get
%\begin{equation*}
$u_t(\Xi_g^{-1}(r))
=
u_t\bigl(\sigma^{-1}(a^{-1}r))
=
u_{t'}(a^{-1}\cdot r)$,
%\end{equation*}
Now choose a lift $f\in O_p^\times$ of $a^{-1}\in (J_p^n)^\times$ and set $\gamma':=f\,\gamma$.
Then the induced map $\overline{\gamma'}:J_p^n(L)\to J_p^n$ satisfies
$\overline{\gamma'}(r)=a^{-1} \cdot \overline{\gamma}(r)$.
Therefore
\begin{equation}\label{eq:final-step}
(u_t\circ \Xi_g^{-1})\circ \overline{\gamma}
=
u_{t'}\circ \overline{\gamma'}
=
u_{t',\gamma'}.
\end{equation}
\end{enumerate}
Combining \eqref{eq:final-step} with \eqref{eq:tay-closure-2} yields $g\cdot u=u_{t',\gamma'}\in \Iso_{\mathrm{Tay}}(J_p^n(L),\F_q^n)$, as claimed.
\end{proof}

\begin{proposition}\label{prop:tay-simply-transitive}
The restriction of the action $\widehat{\Psi}$ to
$\Iso_{\mathrm{Tay}}(J_p^n(L),\F_q^n)$ is simply transitive. That is, it is free and transitive.
\end{proposition}

\begin{proof}
We have to show that the action is transitive and free.
Let us prove that the action is free. Fix $u\in \Iso_{\mathrm{Tay}}(J_p^n(L),\F_q^n)$ and suppose $g\cdot u=u$ for some $g\in G_{n}(p)$. Then $u\circ \Psi_g^{-1}=u$, which implies $\Psi_g=id$ since $u$ is an isomorphism.
Write $g=(a,\sigma)$. Under $\overline{\gamma}:J_p^n(L)\xrightarrow{\sim}J_p^n$, the condition $\Psi_g=id$ is equivalent to
$\Xi_g=id$ on $J_p^n$, i.e.
$a\cdot \sigma(r)=r$ for all $r\in J_p^n$.
Taking $r=1$ gives $a=1$. Then $\sigma(r)=r$ for all $r$, hence $\sigma=id$ and therefore $g=e$.
Thus the stabilizer of $u$ is trivial.

It remains to prove that the action is transitive.
Fix $u,v\in \Iso_{\mathrm{Tay}}(J_p^n(L),\F_q^n)$. Choose a uniformizer $t$ and a trivialization $\gamma$ such that $u=u_{t,\gamma}$,
and similarly choose $t',\gamma'$ such that $v=u_{t',\gamma'}$.
Consider the $\F_q$--linear automorphism
$\phi:=u^{-1}\circ v\ \in\ \Aut_{\F_q}(J_p^n(L))$.
Set $\widetilde\phi:=\overline{\gamma}\circ \phi\circ (\overline{\gamma})^{-1}\ \in\ \Aut_{\F_q}(J_p^n)$.
We have  
$\widetilde\phi=\overline{\gamma}\circ \phi\circ \overline{\gamma}^{-1}= u_t^{-1}\circ u_{t'}\circ (\overline{\gamma'}\circ \overline{\gamma}^{-1})$.
Since $\gamma' = f\,\gamma$ for a unique $f\in O_p^\times$, we have $\overline{\gamma'}\circ \overline{\gamma}^{-1} =\times \overline{f}\quad \text{on } J_p^n$,
Defining 
\[
\sigma:=u_t^{-1}\circ u_{t'}\ \in\ \Aut_{\F_q\text{-alg}}(J_p^n).
\]
and setting $a:=\overline{f}\in (J_p^n)^\times$, we obtain
\[
\widetilde\phi(r)=\sigma(a\,r)=\bigl(\sigma(a)\bigr)\cdot \sigma(r).
\]
Thus $\widetilde\phi=\Xi_{(a',\sigma)}$ with $a':=\sigma(a)\in (J_p^n)^\times$.
Let $g:=(a',\sigma)\in G_n(p)$. Then $\Psi_g=(\overline{\gamma})^{-1}\circ \Xi_g\circ \overline{\gamma}=\phi$, and hence
$g\cdot u \;=\; u\circ \Psi_g^{-1} \;=\; u\circ \phi^{-1} \;=\; v$.
\end{proof}

\subsubsection{On the density of Taylor trivializations}

Fix a trivialization $\gamma_p:L_p\simeq O_p$ and a uniformizer $t_p$ at $p$. The choice of $(t_p,\gamma_p)$ defines a bijection
\begin{equation}\label{eq: Iso-Gln}
\begin{split}
\Phi_{t_p,\gamma_p}:\Iso_{\mathbb{F}_{q}}(J^n_p(L),\mathbb{F}_{q}^n) & \xrightarrow{\ \sim\ }\GL_n(\mathbb{F}_{q})\\
u & \longmapsto u\circ u_{t_p,\gamma_p}^{-1}.
\end{split}
\end{equation}
where $u_{t_p,\gamma_p}$ is defined as in \ref{eq:iso-params}.
On the other hand, $\overline{\gamma_p}$ induces an action 
$$
\Psi:G_n(p)\rightarrow \textrm{Aut}(J^n_p(L))
$$  
through the fundamental action as follows:
$g=(a,\sigma)\mapsto \Psi_g=(\overline{\gamma_p})^{-1}\circ \Xi_g \circ \overline{\gamma_p}$
where $\Xi_g(r)=a\cdot \sigma(r)$ (see \ref{eq:fundamental action}). Recall that the action of $G_n(p)$ on the set of trivializations $\Iso_{\mathbb{F}_{q}}(J^n_p(L), \mathbb{F}_{q}^n)$ is therefore given by 
\begin{equation*}
\begin{split}
\widehat{\Psi}:G_n(p) \times \Iso_{\mathbb{F}_{q}}(J^n_p(L),\mathbb{F}_{q}^n) & \rightarrow \Iso_{\mathbb{F}_{q}}(J^n_p(L),\mathbb{F}_{q}^n)\\
(g,u) & \mapsto (g \cdot u) := u \circ \Psi_g^{-1}
\end{split}
\end{equation*}
\begin{lemma}
The map $\Phi_{t_p,\gamma_p}$ is $\rho_{t_p}$-equivariant, that is,
\[
\Phi_{t_p,\gamma_p}(g\cdot u)
=
\Phi_{t_p,\gamma_p}(u) \cdot \rho_{t_p}(g)^{-1}
\]
for all  $g\in G_n,\ u\in \Iso_{\mathbb{F}_{q}}(J^n_p(L),\mathbb{F}_{q}^n)$, where the right-hand side denotes the standard matrix multiplication.
\end{lemma}

\begin{proof}
Fix $g=(a,\sigma)\in G_n(p)$.
By definition of the representation $\rho_{t_p}$, the matrix representing $\Xi_g$ in the basis $\{1, t_p, \dots, t_p^{n-1}\}$ is exactly $\rho_{t_p}(g)$. In terms of the isomorphism $u_{t_p, \gamma_p}$ defined in \eqref{eq:iso-params}, this is equivalent to
$u_{t_p, \gamma_p} \circ \Psi_g \circ u_{t_p, \gamma_p}^{-1} = \rho_{t_p}(g)$.
The action of $G_n(p)$ on the set of trivializations $\Iso_{\mathbb{F}_{q}}(J^n_p(L), \mathbb{F}_{q}^n)$ is given by the rule $(g \cdot u) := u \circ \Psi_g^{-1}$. Applying  $\Phi_{t_p, \gamma_p}$ to this action, we have
\begin{align*}
\Phi_{t_p, \gamma_p}(g \cdot u) &= (g \cdot u) \circ u_{t_p, \gamma_p}^{-1} \\
&= (u \circ \Psi_g^{-1}) \circ u_{t_p, \gamma_p}^{-1}\\
&= u \circ (\Psi_g^{-1} \circ u_{t_p, \gamma_p}^{-1}) \\
&= u \circ (u_{t_p, \gamma_p}^{-1} \circ u_{t_p, \gamma_p} \circ \Psi_g^{-1} \circ u_{t_p, \gamma_p}^{-1}) \\
&= (u \circ u_{t_p, \gamma_p}^{-1}) \circ (u_{t_p, \gamma_p} \circ \Psi_g^{-1} \circ u_{t_p, \gamma_p}^{-1})\\
&= \Phi_{t_p, \gamma_p}(u) \circ \rho_{t_p}(g)^{-1}=\rho_{t_p}(g)\,\Phi_{t_p,\gamma_p}(u).
\end{align*}
\end{proof}

\begin{lemma}\label{lm:taylor-orbit}
It holds
$\Iso_{\mathrm{Tay}}(J^n_p(L),\mathbb{F}_{q}^n) = \Phi_{t_p,\gamma_p}^{-1}\left( \rho_{t_p}(G_n(p)) \right)$.
That is, $\Iso_{\mathrm{Tay}}(J^n_p(L),\mathbb{F}_{q}^n)$ is a single $G_n(p)$-orbit in $\Iso_{\mathbb{F}_{q}}(J^n_p(L),\mathbb{F}_{q}^n)$. In particular, $\Iso_{\mathrm{Tay}}(J^n_p(L),\mathbb{F}_{q}^n)\subsetneq \Iso_{\mathbb{F}_{q}}(J^n_p(L),\mathbb{F}_{q}^n)$. Furthermore, 
$$
\lim_{n\to\infty}\dfrac{|\Iso_{\mathrm{Tay}}(J^n_p(L),\mathbb{F}_{q}^n)|}{|\Iso_{\mathbb{F}_{q}}(J^n_p(L),\mathbb{F}_{q}^n)|}=0
$$
\end{lemma}
\begin{proof}
The first part follows from Proposition \ref{prop:tay-simply-transitive}.
Now, from (\ref{eq: Iso-Gln}) it follows
$$
|\Iso_{\mathbb{F}_{q}}(J^n_p(L),\mathbb{F}_{q}^n)| = |\GL_n(\mathbb{F}_{q})| = \prod_{i=0}^{n-1}(q^n-q^i).
$$
On the other hand, $|(J^n_p)^\times|=(q-1)q^{n-1}$, $|\Aut_{\mathbb{F}_{q}-alg}(J^n_p)|=(q-1)q^{n-1}$,
and therefore
\[
|G_n|=(q-1)^2 q^{2n-2}.
\]
Thus $|\Iso_{\mathrm{Tay}}(J^n_p(L),\mathbb{F}_{q}^n)|=|G_n|\ll |\GL_n(\mathbb{F}_{q})|=|\Iso_{\mathbb{F}_{q}}(J^n_p(L),\mathbb{F}_{q}^n)| $, and obviously.
$$
\lim_{n\to\infty}\dfrac{|\Iso_{\mathrm{Tay}}(J^n_p(L),\mathbb{F}_{q}^n)|}{|\Iso_{\mathbb{F}_{q}}(J^n_p(L),\mathbb{F}_{q}^n)|}=\lim_{n\to\infty}\dfrac{(q-1)^2 q^{2n-2}}{ \prod_{i=0}^{n-1}(q^n-q^i)}=0
$$
\end{proof}

\subsection{Some remarks}\label{rmk:algebraic-structure}
\begin{enumerate}
\item It is important to recall Remark \ref{rmk: RT}. Given the data $X,D,L$ with $D=n_1 p_1+\cdots+n_s p_s$, one arrives naturally at the $\mathbb{F}_{q}$-linear map
$$
H^0(X,L)\rightarrow L|_D=\oplus_{i=1}^s J^{n_i}_{p_i}(L).
$$
By fixing trivializations $\gamma_i\in \Iso_{\mathbb{F}_{q}}(J^{n_i}_{p_i}(L),\mathbb{F}_{q}^n)$, this yields a $\mathbb{F}_{q}$-linear map
$$
H^0(X,L)\rightarrow L|_D=\mathbb{F}_{q}^{M}, \ M:=\sum_{i=1}^s n_i,
$$
whose image can be regarded as a Goppa code associated to a divisor $D$ having multiplicities different from one, let us say a {multiplicity Goppa code}. The class of differential Goppa codes introduced and explored so far forms a subclass of multiplicity Goppa codes, but a rather special one. This is exactly what Lemma \ref{lm:taylor-orbit} states. While multiplicity Goppa codes are obtained from arbitrary trivializations ($\gamma_i\in \Iso_{\mathbb{F}_{q}}(J^{n_i}_{p_i}(L),\mathbb{F}_{q}^n)$), differential Goppa codes are obtained from Taylor trivializations ($\gamma_i\in \Iso_{\mathrm{Tay}}(J^{n_i}_{p_i}(L),\mathbb{F}_{q}^n)$) which are highly exceptional among generic ones.
\item It is straightforward to show that the set $\Iso^{t}_{\mathrm{Tay}}(J^{n}_{p}(L),\mathbb{F}_{q}^n)$ is an affine variety. Indeed, it is an open subset of an affine subspace of $\Mat_{k \times M}(\mathbb{F}_q)$ defined by the Toeplitz and non-vanishing constant term conditions. 
Furthermore, the action of the Taylor group $G_n(p)$ (which is an affine algebraic group) on the jet spaces is a morphism of varieties. Consequently, the representation $\rho_{t_p}: G_n(p) \to \GL_n(\mathbb{F}_q)$ is a regular map, meaning its matrix entries are polynomial functions of the group coordinates.
\end{enumerate}

\section{Examples}

\subsection{Projective line}

Fix a natural number $k\geq 1$. Let $X = \mathbb{P}^1:=\textrm{Proj}\F_q[x,y]$ be the projective line and fix the invertible sheaf $L = \mathcal{O}((k-1)p_\infty)$ with $p_{\infty}=[1;0]$. Set $t:=x/y$. It is well-known that a basis of $H^0(\mathbb{P}^1,L)$ is given by $\{1,t,\hdots,t^{k-1}\}$.

 \subsubsection{One point case}\label{sec:one point}
Assume  $D = n \cdot p_{\infty}$. In this case, the rational function $t$ induces both, the trivialization $\gamma:L_{p_{\infty}} \rightarrow O_{p_{\infty}}$, $s \mapsto t^{-k+1}s$, and the uniformizer at $p_{\infty}$, $u:=t^{-1}$. For any $t^m$,  the evaluation map $\ev_{t_D,\gamma_D}:H^0(\mathbb{P}^1,L)\rightarrow\F_q^n$ is given by:
\[
ev_{t_D,\gamma_D}(t^m) = \left( D_u^{0}(t^m)(p_{\infty}), D_u^{1}(t^m)(p_{\infty}), \dots, D_u^{n-1}(t^m)(p_{\infty}) \right).
\]
Since $\gamma(t^m)=u^{k-m-1}$, it hods
\[
D_u^{j}(\gamma(t^m))(p_{\infty}) = \delta_{j,k-m-1}:=
\left\{
\begin{array}{ll}
1, & \textrm{ if }j=k-m-1,\\
0, & \textrm{ otherwise.}
\end{array}
\right.
\]
Therefore the generator matrix $G(\infty)\in\mathbb{F}_q^{k\times n}$  is given by either
\[
G(\infty)=\bigl[J_k \ \big|\ 0_{k\times (n-k)}\bigr], \textrm{ if }n\geq k
\]
or 
\[
G(\infty)=
\begin{pmatrix}
0_{(k-n)\times n}\\[2mm]
J_n
\end{pmatrix},
\textrm{ if } n\le k,
\]
where $J_i$ is the $i\times i$ anti-diagonal matrix
\[
J_i=
\begin{pmatrix}
0 & 0 & \cdots & 0 & 1\\
0 & 0 & \cdots & 1 & 0\\
\vdots & \vdots & \ddots & \vdots & \vdots\\
0 & 1 & \cdots & 0 & 0\\
1 & 0 & \cdots & 0 & 0
\end{pmatrix}.
\]
In particular, if $n=k$, $G(\infty)$ is invertible. Equivalently, $ev_{t_D,\gamma_D}$ is an isomorphism.

\subsubsection{Two points case}\label{sec:two points}
Assume now $D = n \cdot p$ with $n\geq k$ and where $p\neq p_{\infty}$ is the point with affine coordinate $\alpha \in \mathbb{F}_q$. 
In this case $L_p=O_p$ so we may choose $id$ as trivialization of $L_p$.
On the other hand, the rational function $u = t - \alpha$ is a uniformizer at $p$. 
For any $t^m$, we already know that the evaluation map $ev_{t_D,\gamma_D}:H^0(\mathbb{P}^1,L)\rightarrow\F_q^n$ is given by:
\[
ev_{t_D,\gamma_D}(t^m) = \left( D_u^{0}(t^m)(\alpha), D_u^{1}(t^m)(\alpha), \dots, D_u^{n-1}(t^m)(\alpha) \right).
\]
Using the binomial expansion $t^m = ( (t-\alpha) + \alpha )^m$, we find that:
\[
D_u^{j}(t^m)(\alpha) = \binom{m}{j} \alpha^{m-j},
\]
with the convention that $\binom{m}{j} = 0$ if $j > m$. The $m$-th row of the generator matrix, which we denote $G(\alpha)$, corresponds to the evaluation of $t^m$. 
Therefore the generator matrix $G(\alpha)\in\mathbb{F}_q^{k\times n}$  has entries
\[
G(\alpha)_{m,j}=\binom{m}{j}\alpha^{m-j},
\qquad \text{with }\binom{m}{j}=0\text{ if }j>m,
\]
and is given by either
\[
G(\alpha)=\bigl[P_k(\alpha)\ \big|\ 0_{k\times (n-k)}\bigr], \textrm{ if }n\geq k,
\]
or
\[
G(\alpha)=
\begin{pmatrix}
P_n(\alpha)\\[2mm]
R_{(k-n)\times n}(\alpha)
\end{pmatrix},
\textrm{ if } n\le k,
\]
where $P_r(\alpha)$ denotes the $r\times r$ lower triangular Pascal matrix scaled by powers of $\alpha$,
\[
P_r(\alpha):=\Bigl(\binom{m}{j}\alpha^{m-j}\Bigr)_{0\le m,j\le r-1}
=
\begin{pmatrix}
1 & 0 & \cdots & 0\\
\alpha & 1 & \cdots & 0\\
\alpha^2 & \binom{2}{1}\alpha & \ddots & \vdots\\
\vdots & \vdots & \ddots & 0\\
\alpha^{r-1} & \binom{r-1}{1}\alpha^{r-2} & \cdots & 1
\end{pmatrix},
\]
and $R_{(k-n)\times n}(\alpha)$ is the $(k-n)\times n$ matrix given  by
\[
R_{(k-n)\times n}(\alpha):=\Bigl(\binom{m}{j}\alpha^{m-j}\Bigr)_{n\le m\le k-1,\;0\le j\le n-1}.
\]
Note that $G(\alpha)$ is full rank. In particular, when $k=n$, the evaluation map is an isomorphism and $G(\alpha)$
 is invertible.
 
 \subsubsection{General case}
Assume now $D=\sum_{i=1}^s n_i p_i$,  $\sum_{i=1}^s n_i=n$ where $p_1=p_\infty$ and for $i\ge 2$ the point $p_i$ has affine coordinate
$\alpha_i\in\mathbb F_q$ (so $p_i=[\alpha_i:1]$).  We order the $n$ coordinates of
$\mathbb F_q^n$ by blocks, the $i$-th block corresponding to the derivatives
$D_{u_i}^j(\cdot)(p_i)$ for $j=0,\dots,n_i-1$.

At $p_\infty$ we take the uniformizer $u_1:=t^{-1}$ and the trivialization $\gamma_1:L_{p_\infty}\to O_{p_\infty},$  $s\mapsto t^{-k+1}s$.
For $i\ge 2$ we take the uniformizer at $p_i$, $u_i:=t-\alpha_i$, and we may use
the canonical trivialization $\gamma_i=\mathrm{id}$ since $p_i\neq p_\infty$. 
Then, the generator matrix $G$ decomposes into blocks
\[
G=\bigl[G{(\infty)}\ \big|\ G{(\alpha_2)}\ \big|\ \cdots\ \big|\ G{(\alpha_s)}\bigr].
\]
where $G{(\infty)}$ and $G{(\alpha_i)}$ are computed as in \S \ref{sec:one point} and \S \ref{sec:two points}. 
 
 Of course, we may consider also the case $p_i\neq p_{\infty}$ for all $i$. In this situation, the generator matrix is just
\[
G=\bigl[G{(\alpha_1)}\ \big|\ G{(\alpha_2)}\ \big|\ \cdots\ \big|\ G{(\alpha_s)}\bigr].
\]

\subsubsection{Genus $0$ differential Goppa codes in the literature}

Our aim now is to show connections between the proposed framework and established constructions in coding theory. 

In the works \cite{xu2023} and \cite{xu2024near}, classical Reed--Solomon codes are extended by adjoining standard basis column vectors of the following form.
\[
V_i = [0,\dots,0,1,0,\dots,0].
\]
These extensions yield Near Maximum Distance Separable (NMDS) codes that admit no realization as classical Goppa codes of genus $0$ or $1$ anymore. However, this is no longer true when considering differential Goppa codes. In this context, the same constructions admit a differential Goppa structure over the projective line $\mathbb{P}^1$ as follows.

 Let $X=\mathbb{P}^1$, let $L=\mathcal{O}_X((k-1)p_\infty)$, and define the divisor
\[
D = 2p_\infty + 2p_0 + \sum_{i=1}^{q-1} p_i .
\]
For $k=4$ or $k=5$, which correspond to the two cases analyzed in \cite{xu2023}, the associated generator matrix takes the following explicit form.
\begin{equation*}
    G_4=
    \begin{bmatrix}
        \alpha_1^{3} & \cdots & \alpha_q^{3} & 1 & 0 & 0\\
        \alpha_1^{2} & \cdots & \alpha_q^{2} & 0 & 1 & 0\\
        \alpha_1     & \cdots & \alpha_q     & 0 & 0 & 1\\
        1            & \cdots & 1            & 0 & 0 & 0
    \end{bmatrix},
    \qquad \alpha_i \neq \alpha_j \in \mathbb{F}_q .
\end{equation*}
Thus, the extended NMDS codes arise naturally from the cohomological description of jet sheaves on $\mathbb{P}^1$, eliminating the need for increasing the genus of the base curve.

A similar phenomenon is observed for the Roth--Lempel codes studied in \cite{han2023roth}. These codes were presented as examples that cannot be obtained from classical Goppa constructions on the projective line or on elliptic curves. Within this framework, however, they admit a  realization as differential Goppa codes on $\mathbb{P}^1$. Again, let $L=\mathcal{O}_X((k-1)p_\infty)$ and define the divisor
\[
D = 2p_\infty + \sum_{i=0}^{q-1} p_i ,
\]
The induced morphism $H^0(\Theta_D)$ yields a generator matrix of the following form.
\begin{equation}\label{M: Roth lemper}
    R_k=
    \begin{bmatrix}
        \alpha_1^{k-1} & \cdots & \alpha_q^{k-1} & 1 & 0\\
        \vdots &  & \vdots & 0 & 1\\
        \vdots &  & \vdots & 0 & 0\\
        \vdots &  & \vdots & \vdots & \vdots\\
        \alpha_1 & \cdots & \alpha_q & 0 & 0\\
        1 & \cdots & 1 & 0 & 0
    \end{bmatrix},
    \qquad \alpha_i \neq \alpha_j \in \mathbb{F}_q .
\end{equation}
This representation demonstrates that the Roth--Lempel family also fits naturally within the differential Goppa formalism.

Although a complete verification of the structural properties requires a detailed analysis comparable to that in the cited works, the general framework presented here clarifies the geometric origin of these constructions. This approach replaces case-by-case or computational arguments with a uniform description in terms of line bundles and divisors on $\mathbb{P}^1$.

\subsection{Elliptic curves}

Assume $\mathrm{char}(\F_q)\neq 2,3$ for simplicity.
Fix natural numbers $1\le k\le n$ and let $E\subset \mathbb{P}^{2}= \textrm{Proj}\ \F_q[X,Y,Z]$ be an elliptic curve in
Weierstrass form 
$$
y^2=x^3+Ax+B, \quad A,B\in\F_q,
$$
where  $x=X/Z$ and $y=Y/Z$ are the affine coordinates on the open subset $Z\neq 0$ of $\mathbb{P}^{2}$. Note that the point at infinity is $e=[0;1;0]$.  Consider now the invertible sheaf
$L= O_E((k-1)e)$. It is well-known that a basis of $H^0(E,L)$ is
\[
\mathcal B_k=\{1,x,\dots,x^{a},\ y,xy,\dots,x^{b}y\},
\]
where 
\[
a:=\Big\lfloor \frac{k-1}{2}\Big\rfloor,\qquad b:=\Big\lfloor \frac{k-4}{2}\Big\rfloor.
\]

%On an elliptic curve one has $\ord_O(x)=-2$ and $\ord_O(y)=-3$. 

%%%%%%%%%%%%%%%%%%%%%%%%%%%%%%%%%%%%%%%%%%%%%%%%%%%%%%%%%%%%%%%%%%%%%%%%%%%%%%%%%%%
\subsubsection{One point case}\label{sec:on point genus one}
%%%%%%%%%%%%%%%%%%%%%%%%%%%%%%%%%%%%%%%%%%%%%%%%%%%%%%%%%%%%%%%%%%%%%%%%%%%%%%%%%%%

Assume $D=n\cdot e$. The rational function $u:=-\frac{x}{y}$ is a uniformizer at $e\in E$. On the other hand, there is an obvious trivialization of $L$ at $e$,
$\gamma:L_e\rightarrow O_{e}$, $s\mapsto u^{\,k-1}s$.
Thus, for every $s\in H^0(E,L)$, we have
\[
ev_{t_D,\gamma_D}(s)=\bigl(D_u^0(\gamma(s))(e),\ D_u^1(\gamma(s))(e),\ \dots,\ D_u^{n-1}(\gamma(s))(e)\bigr)\in\F_q^n.
\]
Since $\ord_e(x)=-2$ and $\ord_e(y)=-3$, we have Laurent expansions at $e$ with respect to $u$
\begin{equation*}
\begin{split}
x&=u^{-2}+(\text{higher powers of }u),\\ 
y&=-u^{-3}+(\text{higher powers of }u),
\end{split}
\end{equation*}
hence
\begin{equation*}
\begin{split}
\gamma(x^m)&=u^{k-1-2m}\bigl(1+\cdots\bigr),\\
\gamma(x^m y)&=u^{k-1-(2m+3)}\bigl(-1+\cdots\bigr).
\end{split}
\end{equation*}
In particular, each $\gamma(b_r)$ has a well-defined lowest term
$\varepsilon_r\,u^{\rho_r}$ with $\varepsilon_r\in\{\pm 1\}$ and $\rho_r\ge 0$,
so that
\[
D_u^j(\gamma(b_r))(e)=0\ \text{ for }j<\rho_r,\qquad
D_u^{\rho_r}(\gamma(b_r))(e)=\varepsilon_r\in\{\pm 1\}.
\]
Reordering the elements of $\mathcal{B}_k$ by increasing order at $e$, we may assume $\rho_1\leq \rho_2 \leq \dots \leq \rho_k$. If we consider the submatrix $G'$ of $G$ formed by the rows $j \in \{\rho_1, \dots, \rho_k\}$, then $G'$ is upper triangular with diagonal entries $\varepsilon_r \in \{\pm 1\}$,

\[
G' = 
\begin{pmatrix}
\varepsilon_1 & * & * & \cdots & * \\
0 & \varepsilon_2 & * & \cdots & * \\
0 & 0 & \varepsilon_3 & \cdots & * \\
\vdots & \vdots & \vdots & \ddots & \vdots \\
0 & 0 & 0 & \cdots & \varepsilon_k
\end{pmatrix},
\]

\noindent
where $*$ denotes (possibly) nonzero entries. In particular, the submatrix $G'$ has full rank $k$, which implies that the original matrix $G$ has rank $k$. Consequently, if $n = k$, the matrix $G$ is invertible.

%%%%%%%%%%%%%%%%%%%%%%%%%%%%%%%%%%%%%%%%%%%%%%%%%%%%%%%%%%%%%%%%%%%%%%%%%%%%%%%%%%%
\subsubsection{Two points case}\label{sec: two points genus one}
%%%%%%%%%%%%%%%%%%%%%%%%%%%%%%%%%%%%%%%%%%%%%%%%%%%%%%%%%%%%%%%%%%%%%%%%%%%%%%%%%%%

Assume now $D=n\cdot p$ where $p=(\alpha,\beta)\in E$ is a rational point with $p\neq e$. For simplicity assume $\beta\neq 0$. Then
\[
v:=x-\alpha
\]
is a uniformizer at $p$, and since $p\neq e$ we may use the identity as trivialization of $L_p$ (i.e., $\gamma=\mathrm{id}$).
The evaluation map is
\[
ev_{t_D,\gamma_D}(f)=\bigl(D_v^0(f)(p),\ D_v^1(f)(p),\ \dots,\ D_v^{n-1}(f)(p)\bigr)\in\F_q^n.
\]
Consider the Taylor expansion of $y$ at $p$ with respect to $v$, 
$$y(v)=\beta+\sum_{r\ge 1} c_r v^r\in \F_q[[v]]$$
Clearly, the coefficients of the expansion are determined by the equation
\[
\bigl(\beta+\sum_{r\ge 1}c_r v^r\bigr)^2=(\alpha+v)^3+A(\alpha+v)+B.
\]
Equating coefficients yields the recursion (with $c_0:=\beta$)
\begin{equation*}
\begin{split}
 2\beta\,c_1&=3\alpha^2+A,\\
 2\beta\,c_2+c_1^2&=3\alpha,\\
2\beta\,c_3+2c_1c_2&=1, \\
2\beta\,c_m+\sum_{r=1}^{m-1}c_rc_{m-r}&=0,  \quad \textrm{ for $m\ge 4$}.
\end{split}
\end{equation*}
The generator matrix $G\in\F_q^{k\times n}$ is defined by $G_{r,j}=D_v^j(b_r)(p)$, with $r=1,\dots,k$, $j=0,\dots,n-1$.
Its entries can be written explicitly as follows. For $m\ge 0$,
\begin{equation}\label{eq:derivatives genus one}
\begin{split}
D_v^j(x^m)(p)&=\binom{m}{j}\,\alpha^{m-j}\\ 
D_v^j(x^m y)(p)&=D_v^j\big((\alpha+v)^m(\beta+\sum_{r\ge 1}c_r v^r)\big)(0)\\
&=\sum_{s=0}^{\min(j,m)} \binom{m}{s}\,\alpha^{m-s}\,c_{j-s}.
\end{split}
\end{equation}
In particular, the first rows (corresponding to $1,x,\dots$) form a scaled
Pascal-type block:
\[
\begin{pmatrix}
1 & 0 & 0 & \cdots \\
\alpha & 1 & 0 & \cdots \\
\alpha^2 & \binom21\alpha & 1 & \cdots \\
\vdots & \vdots & \vdots & \ddots
\end{pmatrix},
\]
while the row for $y$ is
\[
\bigl(\beta,\ c_1,\ c_2,\ \dots,\ c_{n-1}\bigr),
\]
and the row for $xy$ is
\[
\bigl(\alpha\beta,\ \beta+\alpha c_1,\ c_1+\alpha c_2,\ \dots\bigr),
\]
and so on according to the formula \eqref{eq:derivatives genus one}.

%%%%%%%%%%%%%%%%%%%%%%%%%%%%%%%%%%%%%%%%%%%%%%%%%%%%%%%%%%%%%%%%%%%%%%%%%%%%%%%%%%%
\subsubsection{Several points case}
%%%%%%%%%%%%%%%%%%%%%%%%%%%%%%%%%%%%%%%%%%%%%%%%%%%%%%%%%%%%%%%%%%%%%%%%%%%%%%%%%%%

Assume finally that $D=\sum_{i=1}^s n_i p_i$, $\sum_{i=1}^s n_i=n$, where $p_i\in E$ are pairwise distinct rational points and we allow $p_1=e$. 
%We order the $n$ coordinates of $\F_q^n$ by blocks, the $i$-th block being indexed by $j=0,\dots,n_i-1$.
%\paragraph{Local data.}
If $p_i=e$, take $u_i=-x/y$ as uniformizer and trivialize by $\gamma_i(s)=u_i^{k-1}s$ as in the
one point case. If $p_i\neq e$ and $p_i=(\alpha_i,\beta_i)$ with $\beta_i\neq 0$,
take $u_i=x-\alpha_i$ as uniformizer and use the canonical trivialization $\gamma_i=id$.

%\paragraph{Evaluation map.}
%For $f\in H^0(E,L)$ define
%\[
%ev_{D,\gamma}(f)=\Bigl(
%\underbrace{D_{u_1}^0(\gamma_1 f)(P_1),\dots,D_{u_1}^{n_1-1}(\gamma_1 f)(P_1)}_{\text{block }P_1}\ ;\
%\cdots\ ;\
%\underbrace{D_{u_s}^0(\gamma_s f)(P_s),\dots,D_{u_s}^{n_s-1}(\gamma_s f)(P_s)}_{\text{block }P_s}
%\Bigr)\in\F_q^n.
%\]
%
%\paragraph{Generator matrix (block form).}
%Let $\mathcal B_k=\{b_1,\dots,b_k\}$ be the basis above. 
As in the genus $0$ case, the generator matrix decomposes into blocks
\[
G=\bigl[G^{(1)}\ \big|\ G^{(2)}\ \big|\ \cdots\ \big|\ G^{(s)}\bigr],
\qquad G^{(i)}\in\F_q^{k\times n_i}.
\]
If $p_i=e$, the block $G^{(i)}$ is computed as in \S \ref{sec:on point genus one}
If $p_i=(\alpha_i,\beta_i)\neq e$ (with $\beta_i\neq 0$), the block $G^{(i)}$ is computed as in \S \ref{sec: two points genus one}
As in the projective line case, the full generator matrix is obtained by
concatenating the blocks associated to the points in the support of $D$.

\section{Differential Goppa codes II. Duality}

A key result in the classical theory of geometric Goppa codes is that the dual of a Goppa code is again a Goppa code (see \cite{Stichtenoth, TsfasmanBasic}). This important result rests on Serre duality, the Residue Theorem  and on the fact that the residue of meromorphic differentials establishes a canonical trivialization 
$\omega_{X}(p)|_p\simeq \mathbb{F}_q$ (see \cite[\S 3.1.3, pp. 289]{TsfasmanBasic}).
We will show in this section an analogous result, i.e., that the dual of a differential Goppa code is a differential Goppa code as well. To this end, we first need to show that there is a canonical trivialization of $\omega_X (np)|_{np}$ once a uniformizer of $p$ has been fixed.

Throughout this section $X$ is a smooth projective curve of genus $g$ over $\mathbb{F}_q$ and $\omega_X$ its canonical sheaf.

\subsection{The multiplicity bilinear form and the dual code}

Before stating the duality theorem for differential Goppa codes, we need a natural bilinear form to define orthogonality in the context of multiplicity Goppa codes.

Consider $\F_q^M$ and fix a partition of $M$, $H=\{n_1,\dots,n_s\}$ with $\sum_{i=1}^s n_i=M$.
We write vectors in block coordinates
\[
c=(c_{1,0},\dots,c_{1,n_1-1}\,;\ \dots\ ;\,c_{s,0},\dots,c_{s,n_s-1})\in \F_q^M,
\]
and similarly for $c'$. There is a standard bilinear form defined in this context (see \cite{Sharma2014, DoughertyRT}).
\begin{definition}
The {$H$-bilinear form} on $\F_q^M$ is the map
\[
\langle\cdot,\cdot\rangle_H: \F_q^M \times \F_q^M \longrightarrow \F_q
\]
defined by
\[
\langle c,c'\rangle_H \;:=\; \sum_{i=1}^s \ \sum_{j=0}^{n_i-1} c_{i,j}\, c'_{i,n_i-1-j}.
\]
\end{definition}
%\begin{remark}
%The bilinear form $\langle\cdot,\cdot\rangle_H$ can be written in matrix form.  
%For each block $i$, let $J_{n_i}$ denote the $n_i\times n_i$ antidiagonal matrix with ones on the antidiagonal and zeros elsewhere, 
%$$
%J_{n_i}
%=
%\begin{pmatrix}
%0      & 0      & \cdots & 0      & 1 \\
%0      & 0      & \cdots & 1      & 0 \\
%\vdots & \vdots & \ddots & \vdots & \vdots \\
%0      & 1      & \cdots & 0      & 0 \\
%1      & 0      & \cdots & 0      & 0
%\end{pmatrix}.
%$$
%and define
%$J_H := \operatorname{diag}(J_{n_1},\dots,J_{n_s}) \in \textrm{Mat}_{M}(\F_q)$.
%Then $\langle x,y\rangle_H = x^{\mathsf T} J_H y
%\quad \text{for all } x,y\in V$.
%\end{remark}
\begin{lemma}
The bilinear form $\langle\cdot,\cdot\rangle_H$ is non-degenerate.
\end{lemma}

\begin{proof}
Clearly $\langle\cdot,\cdot\rangle_H$ is $\F_q$-bilinear. To prove non-degeneracy, fix $c\in V$ such that $\langle c,c'\rangle_H=0  \text{ for all } c'\in V$.
Fix a block index $i\in\{1,\dots,s\}$ and a coordinate index $j\in\{0,\dots,n_i-1\}$. 
Let $c'\in V$ be the vector with all coordinates equal to $0$ except $c'_{i,n_i-1-j}=1$.
Then, by the definition of $\langle\cdot,\cdot\rangle_H$, every term in the double sum is equal to zero except the one with indices $(i,j)$, and we obtain
\[
0=\langle c,c'\rangle_H = c_{i,j}.
\]
Since $i$ and $j$ were arbitrary, all coordinates of $c$ are zero, hence $c=0$. Therefore the left radical  of $\langle\cdot,\cdot\rangle_H$ is zero, so $\langle\cdot,\cdot\rangle_H$ is non-degenerate.
\end{proof}

\begin{definition}
Let $C\subset \F_q^{M}$ be a linear subspace. The $H$-dual of $C$ is defined by
$$
C^{\bot_H}:=\{c'\in\F_q^M | \, \langle c,c'\rangle_H =0 \  \forall c\in C \}
$$
\end{definition}
%\begin{remark}
%Let $G$ be a generator matrix of $C$. Then, using the matrix representation of
%$\langle\cdot,\cdot\rangle_H$, we obtain
%\[
%C^{\perp_H}
%=
%\{\, y\in V \mid G J_H y = 0 \,\}
%=
%\ker(GJ_H).
%\]
%Equivalently, if $C^\perp$ denotes the dual code with respect to the
%standard inner product, then $C^{\perp_H} = J_H \, C^\perp$.
%\end{remark}

\begin{lemma}
Let $C\subset V$ be a linear code. Then it holds
\begin{enumerate}
\item $\dim C^{\perp_H} = M - \dim C$.
\item $(C^{\perp_H})^{\perp_H} = C$.
%\item $C=C^{\perp_H}$ iff $\dim C = \frac{M}{2}$ and  $G J_H G^{\mathsf T} = 0$.
\end{enumerate}
\end{lemma}
\begin{proof}
Statements (1) and (2) follow from the non-degeneracy of $\langle\cdot,\cdot\rangle_H$.
%Statement (3) is immediate from the definition.
\end{proof}

\subsection{Dual Taylor trivializations}

The dual code of a classical Goppa code is known to be Goppa as well. This fundamental result rests on Serre Duality Theorem and the natural correspondence between trivializations of $L|_p$ and trivializations of $(L^{-1}\otimes\omega_X(p))|_p$. In this section, we generalize this correspondence to the case of multiple points.

\begin{lemma}\label{lm:map theta}
Let $p \in X$ be a rational point and $t$ a uniformizer at $p$. Let $\operatorname{Triv}(L_p)$ denote the set of $O_p$-linear isomorphisms $\gamma: L_p \to O_p$. 
\begin{enumerate}
\item $\operatorname{Triv}(L_p)$ is a torsor under the action of the group of units $O_p^\times$. 
\item For any $\gamma\in \operatorname{Triv}(L_p)$,  the element 
\begin{equation}\label{eq:xigen}
\xi_{gen} := \gamma \otimes (t^{-n} dt)
\end{equation}
is a generator of the $O_p$-module $N_p = L_p^{-1} \otimes \omega_X(np)_p$
\item The map 
\[
\Theta: \operatorname{Triv}(L_p) \longrightarrow \operatorname{Triv}(N_p),
\]
where $\Theta(\gamma)$ is the unique trivialization sending the local generator $\gamma\otimes t^{-n} dt$ to $1$, is equivariant with respect to the inversion automorphism of $O_p^\times$:
$$
\Theta(u\gamma)=u^{-1}\Theta(\gamma).
$$
In particular, it is a bijection.
\end{enumerate}
\end{lemma}

\begin{proof}
1) This is a well-known fact.

2) Since $L_p$ is a locally free $O_p$-module of rank $1$, any isomorphism $\gamma$ is a basis for $L_p^{-1}$. Similarly, $t^{-n} dt$ is a basis for $\omega_X(np)_p$. The tensor product of two bases of rank-$1$ modules is a basis for the tensor product module $N_p$, hence $\xi_{gen}$ is a generator.

3) From 2), we know that  both $\operatorname{Triv}(L_p)$ and $\operatorname{Triv}(N_p)$ are torsors under the action of the group of units $O_p^\times$. 
Now, if we replace $\gamma$ with $\gamma':=u \gamma$ for some $u \in O_p^\times$, the new generator becomes:
\[
\xi'_{gen} = (u \gamma) \otimes (t^{-n} dt) = u (\gamma \otimes t^{-n} dt) = u \xi_{gen}.
\]
Since $\Theta(\gamma')(\xi'_{gen}) = 1$, we must have $\Theta(\gamma') = u^{-1} \Theta(\gamma)$. The map $\Theta$ is thus equivariant with respect to the inversion automorphism of $O_p^\times$. Recall that any equivariant morphism between torsors is necessary a bijection.
\end{proof}

\begin{lemma}
Let $p \in X$ be a rational point and $t$ a uniformizer at $p$. Consider the set  of Taylor trivializations of $J_p^n(L)$ associated to $t$, $\Iso^t_{\mathrm{Tay}}(J^n_p(L),\mathbb{F}_{q}^n)$.
The duality map $\Theta: \operatorname{Triv}(L_p) \to \operatorname{Triv}(N_p)$ descends to a well-defined map between the spaces of Taylor trivializations:
\[
\overline{\Theta}: \Iso^t_{\mathrm{Tay}}(J^n_p(L),\mathbb{F}_{q}^n) \longrightarrow \Iso^t_{\mathrm{Tay}}(J^n_p(N),\mathbb{F}_{q}^n).
\]
Furthermore, the following diagram commutes: 
\[
\begin{CD}
\operatorname{Triv}(L_p) @>\Theta>> \operatorname{Triv}(N_p) \\
@VV \pi_L V @VV \pi_N V \\
 \Iso^t_{\mathrm{Tay}}(J^n_p(L),\mathbb{F}_{q}^n)  @>\overline{\Theta}>> \Iso^t_{\mathrm{Tay}}(J^n_p(N),\mathbb{F}_{q}^n)
\end{CD}
\]
where $\pi_L$ and $\pi_N$ are the natural reduction maps modulo $\mathfrak{m}_p^n$.
\end{lemma}

\begin{proof}
Let $\gamma_1, \gamma_2 \in \operatorname{Triv}(L_p)$ be two trivializations defining the same Taylor trivialization modulo $\mathfrak{m}_p^n$. This implies $\gamma_2 = u \gamma_1$ for some unit $u \in O_p^\times$ such that $u \equiv 1 \pmod{\mathfrak{m}_p^n}$.  By Lemma \ref{lm:map theta} 3), we have $\Theta(\gamma_2) = \Theta(u \gamma_1) = u^{-1} \Theta(\gamma_1)$. Since $u$ is of the form $1 + m$ for some $m \in \mathfrak{m}_p^n$, its inverse in the local ring $O_p$ also satisfies $u^{-1} \equiv 1 \pmod{\mathfrak{m}_p^n}$.  Consequently, $\pi_N(\Theta(\gamma_2)) = \pi_N(\Theta(\gamma_1))$, which proves that $\overline{\Theta}$ is well-defined. The commutativity of the diagram follows directly from the construction.
\end{proof}

\begin{lemma}\label{prop: charac theta(gamma)}
Let $p \in X$ be a rational point and $t$ a uniformizer at $p$
Let $\gamma \in \operatorname{Triv}(L_p)$ be a local trivialization and let $\Theta(\gamma): N_p \xrightarrow{\sim} O_p$ be the dual trivialization.
For any section $\xi \in N_p$, the coordinate vector $(g_0, \dots, g_{n-1}) \in \mathbb{F}_q^n$ of $\overline{\xi}:=\xi \mod{\mathfrak{m}_p^n\in J_p^n(N)}$ under the Taylor trivialization $\overline{\Theta(\gamma)}$ is given by
\[
g_k = \operatorname{Res}_p(t^{n-1-k} \eta_\xi), \quad \text{for } k=0, \dots, n-1,
\]
where $\eta_\xi \in \omega_X(np)_p$ is the differential satisfying the relation $\xi = \gamma \otimes \eta_\xi$.
\end{lemma}

\begin{proof}
By definition of $N_p$ as a rank-1 $O_p$-module, the section $\xi$ can be expressed uniquely as $\xi = g \cdot \xi_{gen}$ for some $g \in O_p$ (see \ref{eq:xigen}). 
On the other hand, $\Theta(\gamma)$ is uniquely determined by the condition $\Theta(\gamma)(\xi_{gen}) = 1$ (see Lemma \ref{lm:map theta}).
Applying the trivialization $\Theta(\gamma)$ and using its $O_p$-linearity, we get
$\Theta(\gamma)(\xi) = \Theta(\gamma)(g \cdot \xi_{gen}) = g \cdot \Theta(\gamma)(\xi_{gen}) = g$. Thus, the image of $\overline{\Theta(\gamma)}(\overline{\xi})$  is the vector $(g_0, \dots, g_{n-1})$ consisting of the first $n$  coefficients of the Taylor expansion of the function $g$ at $p$ with respect to the uniformizer $t$.  To relate these coefficients to the residues of $\eta_\xi$, we compare the two representations of $\xi$
\begin{equation}\label{eq:presenntation eta}
\gamma \otimes \eta_\xi = g \cdot (\gamma \otimes t^{-n} dt) = \gamma \otimes (g t^{-n} dt).
\end{equation}
This implies the identity of meromorphic differentials $\eta_\xi = g t^{-n} dt$. Substituting the Taylor expansion $g(t) = \sum_{j=0}^{\infty} g_j t^j$ into this expression yields
\[
\eta_\xi = \left( \sum_{j=0}^{\infty} g_j t^j \right) t^{-n} dt = \sum_{j=0}^{\infty} g_j t^{j-n} dt.
\]
On the other hand, for an arbitrary integer $m \geq 0$, we have
\[
\operatorname{Res}_p(t^m \eta_\xi) = \operatorname{Res}_p \left( \sum_{j=0}^{\infty} g_j t^{j-n+m} dt \right).
\]
By the definition of the residue, the only term that contributes is the one where the exponent of $t$ is $-1$. This happens when $j-n+m = -1$, or equivalently, $j = n-1-m$. Therefore $\Res_p(t^m \eta_\xi) = g_{n-1-m}$. Setting $k = n-1-m$, which implies $m = n-1-k$, we obtain:
\[
g_k = \operatorname{Res}_p(t^{n-1-k} \eta_\xi).
\]
\end{proof}

\subsection{Duality theorem}

Consider the geometric data $(X,L,D)$ as in \S\ref{sec:definition} and set $M=\sum n_i$. We then have the short exact sequence
$$
0\rightarrow L(-D) \rightarrow L \rightarrow L|_D\rightarrow 0
$$
Taking global sections, we get
$$
\cdots \rightarrow H^0(X,L) \rightarrow L|_D\rightarrow H^1(X,L(-D))\rightarrow \cdots
$$
Here, $L|_{D}$ is just $\oplus_{i=1}^s J_{p_i}^{n_i-1}(L)$ (see \eqref{eq:fibers of jets}).
Let $L|_D\simeq \mathbb{F}_q^{M}$ be a Taylor trivialization given by the data $(\{t_i\},\gamma_{D})$. Then, we already know that
the image of 
$$H^0(X,L) \rightarrow L|_D\simeq \mathbb{F}_q^{M}$$ 
is precisely the differential Goppa code $C(X,L,D,t_D,\gamma_{D})$ (see Remark \ref{rmk: RT}). On the other hand, we may consider the corresponding exact sequence in cohomology for the invertible sheaf $N=L^{-1}\otimes\omega_X(D)$:
$$
 \cdots\rightarrow H^0(X,L^{-1}\otimes\omega_X(D)) \rightarrow (L^{-1}\otimes\omega_X(D))|_D \rightarrow H^1(X,L^{-1}\otimes \omega_X) \rightarrow \cdots \\
$$

\begin{definition}
The dual differential Goppa code associated to the data $(X,L,D,t_D,\gamma_{D})$ is defined as the differential Goppa code 
$$
C(X,N,D,t_D,\Theta(\gamma_D)) \ \textrm{ with } N=L^{-1}\otimes \omega_X(D).
$$ 
where $\Theta(\gamma_D):=\{ \Theta(\gamma_i)\}$
We denote by $\Res_{t_D,\gamma_D}$ the linear map
$$
H^0(X,L^{-1}\otimes\omega_X(D)) \rightarrow (L^{-1}\otimes\omega_X(D))|_D\simeq \F_q^M.
$$
whose image is $C(X,N,D,t_D,\Theta(\gamma_D))$.
\end{definition}

\begin{proposition}\label{lem:hasse-residue-convolution}
For sections $s \in H^0(X,L)$ and $\xi \in H^0(X,N)$, let $\langle s, \xi \rangle \in H^0(X, \omega_X(D))$ be the section induced by the natural contraction $L \otimes L^{-1} \to O_X$. Fix an index $i$, a uniformizer $t_i$ at $p_i$, and a local trivialization $\gamma_i: L_{p_i} \xrightarrow{\sim} O_{p_i}$. Then:
\[
\operatorname{Res}_{p_i}(\langle s, \xi \rangle) = \sum_{j=0}^{n_i-1} D_{t_i}^{(j)}\bigl(\gamma_i(s)\bigr)(p_i) \cdot \operatorname{Res}_{p_i} \Bigl(  t_i^{n_i-1-j} \eta_i \Bigr).
\]
where $\eta \in \omega_X(np)_{p_i}$ is the differential satisfying $\xi = \gamma_i \otimes \eta_i$ in $O_{p_i}$.
\end{proposition}

\begin{proof}
Given $\gamma_i$ and $t_i$, we have $\Theta(\gamma_i)$. These trivializations determine locas basis of $L_{p_i}$ and $N_{p_i}$. Concretely, $e_i=\gamma_i^{-1}(1)$ is a basis of $L_{p_i}$ and $\gamma_i\otimes t^{-n}dt$ is a basis of $N_{p_i}$. Then, we may express $s$ and $\xi$ locally at $p_i$  as:
$s = f_i e_i$ and $\quad \xi = g_i \left( \gamma_i \otimes \frac{dt_i}{t_i^{n_i}} \right)$,
where $f_i, g_i \in O_{p_i}$. The contraction $\langle s, \xi \rangle$ in $\omega_X(D)_{p_i}$ is then given by
\[
\langle s, \xi \rangle = f_i g_i \frac{dt_i}{t_i^{n_i}}.
\]
The residue $\operatorname{Res}_{p_i}(\langle s, \xi \rangle)$ corresponds to the coefficient of $t_i^{-1}$ in the Laurent expansion of $f_i g_i /t_i^{n_i}$, which is equivalent to the $(n_i-1)$-th coefficient in the Taylor expansion of $f_i(t_i)g_i(t_i)$. Let $f_i(t_i) = \sum_{a=0}^\infty f_{i,a} t_i^a$ and $g_i(t_i) = \sum_{b=0}^\infty g_{i,b} t_i^b$ be the respective expansions. By the Cauchy product, we have:
\begin{equation}\label{eq:convolution}
\operatorname{Res}_{p_i}(\langle s, \xi \rangle) = \sum_{j=0}^{n_i-1} f_{i,j} g_{i,n_i-1-j}.
\end{equation}
By definition, the coefficients of the Taylor expansion of $f_i$ are given by the Hasse derivatives
$f_{i,j} = D_{t_i}^{(j)}(f_i)(p_i) = D_{t_i}^{(j)}(\gamma_i(s))(p_i)$. From Lemma \ref{prop: charac theta(gamma)} it follows $g_{i,n_i-1-j} = \Res_{p_i} ( g_i(t_i) \frac{dt_i}{t_i^{n_i-j}} ) = \Res_{p_i} (t_i^{n_i-1-j} \eta_i )$.
Substituting these identities into \eqref{eq:convolution} completes the proof.
\end{proof}

\begin{theorem}\label{th:duality}
Let $(X, L, D)$ be as above and let $ (t_D, \gamma_{D})$ be the chosen Taylor data. Let $C = C(X, L, D, t_D, \gamma_{D})$ be the corresponding differential Goppa code. The dual code of $C$ with respect to the $H$-bilinear form on $\mathbb{F}_q^M$ is precisely the dual differential Goppa code:
\[
C^{\perp_H} = C(X, N, D, t_D, \Theta(\gamma_D)).
\]
where $\Theta(\gamma_D):=\{\Theta(\gamma_1),\hdots,\Theta(\gamma_s)\}$.
\end{theorem}

\begin{proof}
Let $s \in H^0(X, L)$ and $\xi \in H^0(X, N)$. Let $c =ev_{t_D,\gamma_D}(s)$ and $c' = \Res_{t_D,\gamma_D}(\xi)$ be the corresponding codewords.  By the definition of the $H$-bilinear form on $\mathbb{F}_q^M$, we have:
\[
\langle c, c' \rangle_H = \sum_{i=1}^s \langle c_i, c'_i \rangle_{H, n_i}, \  \langle c_i, c'_i \rangle_{H, n_i}:=\sum_{j=0}^{n_i-1} c_{i,j}\, c'_{i,n_i-1-j}. %\langle c_i, c'_i \rangle_{H, n_i},
\]
where $c_i = (c_{i,0}, \dots, c_{i, n_i-1})$ and $c'_i = (c_{i,0}, \dots, c_{i, n_i-1})$. According to the Lemma \ref{prop: charac theta(gamma)}, the coordinates of the dual section satisfy $c'_{i, n_i-1-j} = \operatorname{Res}_{p_i}(t_i^j \eta_{\xi,i})$, where $\eta_{\xi,i}$ is the local differential representative of $\xi$ at $p_i$ such that $\xi = \gamma_i \otimes \eta_{\xi,i}$. 
On the other hand, the coefficients $c_i$ are given by the Hasse derivatives (see Remark \ref{rmk:code-words}):
\[
c_{i,j} = D_{t_i}^{(j)}(f_i)(p_i) = D_{t_i}^{(j)}(\gamma_i(s))(p_i).
\]
Substituting these into the local $H$-pairing:
\[
\langle c_i, c'_i \rangle_{H, n_i} = \sum_{j=0}^{n_i-1} c_{i,j} c'_{i, n_i-1-j} = \sum_{j=0}^{n_i-1} D_{t_i}^{(j)}(\gamma_i(s))(p_i) \operatorname{Res}_{p_i}(t_i^j \eta_{\xi,i}) %= \operatorname{Res}_{p_i} \left( \sum_{j=0}^{n_i-1} a_{i,j} t_i^j \eta_{\xi,i} \right).
\]
From Proposition \ref{lem:hasse-residue-convolution}, it follows that the last term is the residue of $\langle s, \xi \rangle$ at $p_i$. Therefore
$\langle c, c' \rangle_H = \sum_{i=1}^s \operatorname{Res}_{p_i}(\langle s, \xi \rangle)$.
By the Residue Theorem \eqref{eq: residue theorem}, the sum of residues of a global meromorphic differential over all points of the curve $X$ is zero. Since $\langle s, \xi \rangle$ is regular outside the support of $D$, we have $\sum_{i=1}^s \operatorname{Res}_{p_i}(\langle s, \xi \rangle) = 0$, proving that $C(X, N, D, t_D, \Theta(\gamma_D)) \subseteq C^{\perp_H}$.

To prove the equality, we compare the dimensions. From the long exact sequences in cohomology, Serre Duality, and Lemma~\ref{lm:dimension Goppa} we know:
\begin{equation*}
\begin{split}
\dim C(X, L, D, t_D, \gamma_{D}) &= \dim H^0(X, L) - \dim H^0(X, L(-D))\\
\dim C(X, N, D, t_D, \Theta(\gamma_D)) &= \dim H^0(X, N) - \dim H^0(X, N(-D))\\
&= \dim H^1(X, L(-D)) - \dim H^1(X, L)
\end{split}
\end{equation*}
Therefore,
\begin{equation*}
\begin{split}
\dim C(X, N, D, t_D, \Theta(\gamma_D))=&M- \dim H^0(X, L) - \dim H^0(X, L(-D)) \\
=&C^{\perp_H}.
\end{split}
\end{equation*}
which completes de proof.
\end{proof}

\section{The block and Hamming distance. Design via local parameters}\label{sec:blk-hamming-design}

For the rest of this section, we fix a smooth projective curve $X$ of genus $g$ over the finite field $\mathbb{F}_q$, $L$ an invertible sheaf and $D=n_1p_1+\cdots +n_sp_s$ an effective divisor of $X$.  
Let us consider the linear map
$$
H^0(\Theta_D) \;:=\; H^0(X,L)\rightarrow H^0(X,L|_D)=\bigoplus_{i=1}^s L_{p_i}/\fm_{p_i}^{n_i}L_{p_i}
$$
induced by the canonical surjection $L\twoheadrightarrow L|_D$ (Recall Definition \ref{def: theta map}).
For a choice of uniformizers and trivializations $(t_D,\gamma_D)$, we have an associated Taylor trivialization
$$
\overline{\gamma}:H^0(X,L|_D)\simeq \F_q^M,
$$ 
which yields the differential Goppa code (Definition \ref{def:Goppa})
\[
C(X,L,D,t_D,\gamma_D)=\im(\ev_{t_D,\gamma_D})\subseteq \F_q^M.
\]
By Lemma~\ref{lm:dimension Goppa}, the kernel of $\ev_{t_D,\gamma_D}$ is $H^0(X,L(-D))$, hence
\[
k:=\dim_{\F_q} C(X,L,D,t_D,\gamma_D)=\ell(L)-\ell(L(-D)),
\]
independently of $(t_D,\gamma_D)$. While the length and dimension of a differential Goppa code are independent on the chosen $(t_D,\gamma_D)$, this is no longer true for the minimum Hamming distance. 
Our aim in this section is to shed light on the dependence of the minimum Hamming distance on the variation of local data $(t_D,\gamma_D)$.

\subsection{Definition and fundamental properties of the block distance}\label{subsec:blk-def-props}

The space $H^0(X,L|_D)\simeq\bigoplus_{i=1}^s L_{p_i}/\fm_{p_i}^{n_i}L_{p_i}$ has a natural decomposition into $s$ summands, one for each point of $\Supp(D)$.
This leads to a canonical notion of ``block support'', independent of any choice of coordinates.
\begin{definition}\label{def:block-metric}
Let $V:=\bigoplus_{i=1}^s V_i$ be a direct sum of $\F_q$--vector spaces. For $v=(v_1,\dots,v_s)\in V$, we define the block weight of $v$ by
\[
\wt_{\mathrm{blk}}(v):=\#\{\,i:\ v_i\neq 0\,\}.
\]
We define the block distance on $V$ by
\[
d_{\mathrm{blk}}(v,w):=\wt_{\mathrm{blk}}(v-w).
\]
\end{definition}

\begin{lemma}
The function $d_{\mathrm{blk}}$ is a metric on $V$.
\end{lemma}
\begin{proof}
Symmetry and positivity are clear. Moreover, $d_{\mathrm{blk}}(v,w)=0$ if and only if $v-w=0$, i.e.\ $v=w$.
For the triangle inequality, fix $u,v,w\in V$. 
Since $(u-w)_i= (u-v)_i + (v-w)_i$, the condition $(u-w)_i\neq 0$ implies $(u-v)_i\neq 0$ or $(v-w)_i\neq 0$. Therefore,
\[
\{i:\ (u-w)_i\neq 0\}\subseteq \{i:\ (u-v)_i\neq 0\}\cup \{i:\ (v-w)_i\neq 0\}.
\]
Taking cardinalities gives $d_{\mathrm{blk}}(u,w)\le d_{\mathrm{blk}}(u,v)+d_{\mathrm{blk}}(v,w)$.
\end{proof}

\begin{definition}\label{def:dblk-code}
Let $M\in\mathbb{N}$ be a natural number and $M=\sum_{i=1}^{s}n_i$ a decomposition. Consider on $\F_q^M$ the block distance associated to the decomposition $\F_q^M=\oplus_{i=1}^{s}\F_q$. The block distance of a linear code $C\subset \F_q^M$ is defined  by
\[
d_{\mathrm{blk}}(C)\;:=\;\min\bigl\{\wt_{\mathrm{blk}}(c):\ c\in C \bigr\}.
\]
\end{definition}

\begin{remark}
By Theorem~\ref{th: taylor parameters}, changes of Taylor data act blockwise by invertible transformations. Hence,
given a differential Goppa code $C(X,L,D),t_D,\gamma_D)\subset \F_q^M$,
the quantity $d_{\mathrm{blk}}(C(X,L,D),t_D,\gamma_D))$ depends only on $(X,L,D)$, i.e., it is independent of all choices $(t_D,\gamma_D)$ used to identify $H^0(X,L|_D)$ with $\F_q^M$. Therefore, it is natural to use the notation 
$$
d_{\mathrm{blk}}(X,L,D)
$$ 
instead of $d_{\mathrm{blk}}(C(X,L,D),t_D,\gamma_D))$.
\end{remark}

\begin{lemma}\label{lm:blk-vs-hamming}
For every choice $(t_D,\gamma_D)$, the Hamming distance $d_H$ of $C(X,L,D,t_D,\gamma_D)\subseteq \F_q^M$ satisfies
\[
 d_{\mathrm{blk}}(X,L,D) \le d_H\big(C(X,L,D,t_D,\gamma_D)\big) \le n\, d_{\mathrm{blk}}(X,L,D)
\]
where $n:=\max_i \{n_i\}$.
\end{lemma}
\begin{proof}
Fix $0\neq c\in C(X,L,D,t_D,\gamma_D)$ and set $\tilde c:=\overline{\gamma}^{-1}(c) \in H^0(X,L|_D)$.
If $\tilde c$ has $r$ nonzero blocks, then each such block is a nonzero vector in $\F_q^{n_i}$ under the chosen identification. Hence, it contributes at least one nonzero coordinate to $c$. Therefore $\wt_H(c)\ge r=\wt_{\mathrm{blk}}(\tilde c)$, and taking minima gives $d_H\ge d_{\mathrm{blk}}$. For the second inequality, if $\tilde c$ has $r$ nonzero blocks, then $c$ has support contained in the union of these blocks, of total size at most $\sum_{j=1}^r n_{i_j}\le r\,n$. Hence $\wt_H(c)\le n \,\wt_{\mathrm{blk}}(\tilde c)$ and taking minima we get the result.
\end{proof}

\begin{remark}
Assume  $n_1=\cdots=n_s=n$. Endow $\F_q^{sn}=(\F_q^n)^s$
with the block Rosenbloom--Tsfasman weight (see \cite{DoughertyRT, Gopinadh2021,RT, Sharma2014})
$$
\wt_{RT}\bigl((v^{(1)},\dots,v^{(s)})\bigr):=\sum_{i=1}^s \rho\bigl(v^{(i)}\bigr),
$$
where, for any $w\in\F_q^n$, we define
$$
\rho(w)=\begin{cases}
0,&w=0,\\
\max\{j:\ w_j\neq 0\},&w\neq 0.
\end{cases}
$$
Then $\rho(w)\ge 1$ for all $w\neq 0$, so $\wt_{RT}(v)\ge \wt_{\mathrm{blk}}(v)$ for all $v$. In particular,
$$
d_{RT}\big(C(X,L,D,t_D,\gamma_D)\big)\ \ge\ d_{\mathrm{blk}}(X,L,D),
$$
for every choice $(t_D,\gamma_D)$, where $d_{RT}$ denotes the distance with respect to the Rosenbloom--Tsfasman weight.
\end{remark}

\begin{remark}
Assume $n_1=\cdots=n_s=n$ and identify $\F_q^{sn}$ with the space  of matrices $\textrm{Mat}_{s\times n}(\F_q)$ by grouping the $n$ coordinates at each point into the $i$--th row.
Then, any $c\in C(X,L,D,t_D,\gamma_D)$ can be regarded as a matrix $A_c\in \textrm{Mat}_{s\times n}(\F_q)$ in this way. The rank weight is then defined as (see \cite{Gabidulin2025})
$$
\rk(c)=\textrm{rank}(A_c).
$$
By Theorem \ref{th: taylor parameters}, it follows that $\rk(c)$ does not depend on the particular choice of $(t_D,\gamma_D)$ so we will use the notation $d_{rk}(X,L,D)$ instead of $d_{rk}\big(C(X,L,D,t_D,\gamma_D)\big)$.
On the other hand, it is clear that for any matrix $A\in \textrm{Mat}_{s\times n}(\F_q)$ we have
$$
\rk(A)\le \#\{\text{nonzero rows of }A\}=\wt_{\mathrm{blk}}(A).
$$
Therefore, for every choice $(t_D,\gamma_D)$, it holds
\[
\begin{split}
d_{rk}(X,L,D) & \le\ d_{\mathrm{blk}}(X,L,D),\\
d_{rk}(X,L,D) & \le \min\{s,n\},
\end{split}
\]
where $d_{rk}$ is the  distance with respect to the rank weight.
\end{remark}

\subsection{Achievability of $d_{\mathrm{blk}}$ and an explicit construction}\label{subsec:blk-achievable}

In this subsection we show that the block distance is the smallest Hamming distance one can get by varying the Taylor trivializations of $L|_D$. Indeed, we provide a constructive procedure to find the local data $(t_D,\gamma_D)$ giving the block distance .

\begin{lemma}\label{lm:local-sparsify}
Let $p\in X$ be a rational point, $t$  a uniformizer at $p$, and $\gamma:L_p\simeq O_p$ a trivialization.
Fix $n\ge 1$ and let $s\in H^0(X,L)$ be such that $\gamma(s)(t)\neq 0 \mod{t^n}$. Set
$m:=\min\{l\in\{0,\hdots,n-1\}| \ D_t^{l}(\gamma(s))\neq 0\}$.
Then, there exists $f\in O_p^{\times}$ such that $\gamma':=f \gamma$ and  satisfies
\[
\gamma'(s)(t) = t^m \pmod{t^n}.
\]
Equivalently, $\big(D_t^0(\gamma'(s)),\dots,D_t^{n-1}(\gamma'(s))\big)=(0,\dots,0,1,0,\dots,0)$, with the $1$ in position $m$.
\end{lemma}
\begin{proof}
Set
$v:=\big(D_t^0(\gamma(s)),\dots,D_t^{n-1}(\gamma(s))\big)\in \F_q^n$, and $m:=\min\{\,\ell\in\{0,\dots,n-1\}\mid v_\ell\neq 0\,\}$.
 By Lemma~\ref{lm:change of trivialization at a point}, for every  $f\in O_p^\times$ and for $\gamma':=f\gamma$ one has
\begin{equation}\label{eq:short-mat}
\big(D_t^0(\gamma'(s)),\dots,D_t^{n-1}(\gamma'(s))\big)=v\,S_p(f),
\end{equation}
where $S_p(f)\in \GL_n(\F_q)$ is lower triangular with entries $(S_p(f))_{i,j}=D_t^{\,i-j}(f)$ for $j\le i$. %We seek $f\in O_p^\times$ such that $vS_p(f)=e_m$.
Now, for any $f\in O_p^{\times}$ we may write 
$f(t)= \sum_{r=0}^{n-1} b_r t^r \pmod{t^n}$, with $b_0\neq 0$,
Then $D_t^r(f)=b_r$ for $0\le r\le n-1$, so $S_p(f)$ is determined by $b_0,\dots,b_{n-1}$.
Set $w:=vS_p(f)=(w_0,\dots,w_{n-1})$. A direct multiplication gives, for each $k$,
$w_k=\sum_{i=k}^{n-1} v_i\,b_{i-k}$.
Since $v_0=\cdots=v_{m-1}=0$, we have $w_k=0$ for $k<m$.
Thus the condition $w=e_m$ is equivalent to the system
\begin{equation}\label{eq:short-system}
w_m=1,\qquad w_{m+1}=w_{m+2}=\cdots=w_{n-1}=0.
\end{equation}
Put $r:=n-1-m$ and consider the indeterminates $b_0,\dots,b_r$ (note that $w_k$ only involves $b_0,\dots,b_{n-1-k}$).
For $u=0,\dots,r$ (i.e.\ $k=m+u$), we have
$w_{m+u}=\sum_{j=0}^{r-u} v_{m+u+j}\,b_j$.
This gives an upper triangular linear system in the variables $b_0,\dots,b_r$:
\[
\begin{pmatrix}
w_m\\
w_{m+1}\\
\vdots\\
w_{n-2}\\
w_{n-1}
\end{pmatrix}
=
\begin{pmatrix}
v_m     & v_{m+1} & \cdots & v_{n-2} & v_{n-1}\\
0       & v_m     & \cdots & v_{n-3} & v_{n-2}\\
\vdots  & \vdots  & \ddots & \vdots  & \vdots\\
0       & 0       & \cdots & v_m     & v_{m+1}\\
0       & 0       & \cdots & 0       & v_m
\end{pmatrix}
\begin{pmatrix}
b_0\\
b_1\\
\vdots\\
b_{r-1}\\
b_r
\end{pmatrix}, \  r=n-1-m.
\]
Since $v_m\neq 0$, the system admits a solution $(b_0,\dots,b_r)\in \F_q^{r+1}$ with $b_0\neq 0$. Now, choose any lift $f\in O_p^\times$ whose Taylor expansion at $p$ respect to $t$ agrees with the above coefficients modulo $t^n$. Then \eqref{eq:short-mat} and \eqref{eq:short-system} give $vS_p(f)=e_m$, i.e. $\big(D_t^0(\gamma'(s)),\dots,D_t^{n-1}(\gamma'(s))\big)=e_m$. Finally, by definition of $D_t^\ell$, this is equivalent to $\gamma'(s)(t)= t^m \pmod{t^n}$.
\end{proof}

\begin{theorem}\label{th:dblk-achievable}
There exist trivializations $\gamma_D'=(\gamma_1',\dots,\gamma_s')$ (with uniformizers fixed) such that
\[
d_H\big(C(X,L,D,t_D,\gamma_D')\big)=d_{\mathrm{blk}}(X,L,D).
\]
That is, the minimum Hamming distance achievable by varying $(t_D,\gamma_D)$ is exactly $d_{\mathrm{blk}}(X,L,D)$.
\end{theorem}
\begin{proof}
Fix the uniformizers $t_D$ and a section $s_0\in H^0(X,L)\setminus H^0(X,L(-D))$ such that $\wt_{\mathrm{blk}}(\ev_{t_D,\gamma_D}(s_0))=d_{\mathrm{blk}}(X,L,D)$.
Let $S\subseteq\{1,\dots,s\}$ be the set of indices $i$ such that the $i$--th block of $\ev_{t_D,\gamma_D}(s_0)$ is nonzero.
Thus $|S|=d_{\mathrm{blk}}(X,L,D)$, and for $i\notin S$ we have $\gamma_i(s_0)\in \fm_{p_i}^{n_i}$.
For each $i\in S$, apply Lemma~\ref{lm:local-sparsify} at $p_i$ with $n=n_i$ to find a unit $f_i\in O_p^{\times}$ such that the new trivialization  $\gamma_i':=f_i\,\gamma_i$ 
satisfies $(D_{t_i}^0(\gamma_i'(s_0)),\dots,D_{t_i}^{n_i-1}(\gamma_i'(s_0)))= e_{m_i}$ with  $e_{m_i}$ a standard basis vector in $\F_q^{n_i}$.
For $i\notin S$, set $\gamma_i':=\gamma_i$. Consider the code 
$$C':=C(X,L,D,t_D,\gamma_D').$$
The codeword $c_0':=ev_{t_D,\gamma_D'}(s_0)\in C'$ has exactly one nonzero coordinate in each block indexed by $S$, and it is zero on blocks $i\notin S$.
Therefore $\wt_H(c_0')=|S|=d_{\mathrm{blk}}(X,L,D)$, which implies
$d_H(C')\le d_{\mathrm{blk}}(X,L,D)$. On the other hand, Lemma~\ref{lm:blk-vs-hamming} gives 
$d_H(C')\ge d_{\mathrm{blk}}(X,L,D)$. Hence equality holds.
\end{proof}

%\begin{remark}[Algorithmic construction]\label{rmk:algorithm-dblk}
%The proof of Theorem~\ref{th:dblk-achievable} yields a concrete procedure.
%
%\smallskip
%\noindent\textbf{Input:} $(X,L,D)$ and an initial choice $(t_D,\gamma_D)$.
%
%\noindent\textbf{Step 1:} Compute $d_{\mathrm{blk}}(X,L,D)$ by searching for the largest $T\subseteq\{1,\dots,s\}$ such that
%\[
%\ell\!\left(L-\sum_{i\in T} n_i p_i\right)>\ell(L-D),
%\]
%and then setting $d_{\mathrm{blk}}=s-|T|$ (equivalently, find $s_0$ with maximal number of vanishing blocks).
%
%\noindent\textbf{Step 2:} Choose $s_0\in H^0\!\left(L-\sum_{i\in T}n_ip_i\right)\setminus H^0(L-D)$.
%
%\noindent\textbf{Step 3:} For each $i$ with nonzero $i$--th block, write the truncated expansion
%$\gamma_i(s_0)(t_i)\bmod t_i^{n_i}$ and compute its unit inverse $u_i(t_i)$ modulo $t_i^{n_i-m_i}$ as in Lemma~\ref{lm:local-sparsify}.
%
%\noindent\textbf{Step 4:} Set $\gamma_i':=u_i\,\gamma_i$ for those $i$; leave the other $\gamma_i$ unchanged.
%
%\noindent\textbf{Output:} $(t_D,\gamma_D')$ such that $d_H(C(X,L,D,t_D,\gamma_D'))=d_{\mathrm{blk}}(X,L,D)$.
%\end{remark}

\subsection{Existence of parameters with $d_H \ge d$}\label{subsec:large-q-existence}

We now address the following design problem: given the geometric data $(X,L,D)$ and a natural number $d$, we seek to determine conditions that ensure the existence of  Taylor trivializations $(t_D, \gamma_D)$ such that the resulting differential Goppa code satisfies $d_H \ge d$.

By the Singleton bound, a necessary condition is $d \le M-k+1$. Furthermore, Lemma~\ref{lm:blk-vs-hamming} implies that $d \ge d_{\mathrm{blk}}(X,L,D)$ is a natural lower bound for the minimum distance achievable within this framework. Therefore, $d$ must satisfy
$$
d_{\mathrm{blk}}(X,L,D) \le d \le M-k+1.
$$

Recall from Definition~\ref{def:taylor-group} that the Taylor group $G_n(p)$ acts on the set of local parameters. In this subsection, we simplify the search space by fixing  uniformizers $t_D = (t_1, \dots, t_s)$ and allowing only the trivializations to vary. 
This requires to study the orbit of an initial choice $\gamma_D$ under the action of the subgroup of local rescalings:
\[
\prod_{i=1}^s (J^{n_i}_{p_i})^\times \subset \prod_{i=1}^s G_{n_i}(p_i).
\]
Since $(J^n_p)^\times$ acts linearly on the jet space via the lower triangular matrices $S_p(f)$ (see Lemma \ref{lm:action local parameters}), this restriction allows us to treat the code parameters as coordinates in an affine variety.

\begin{definition}\label{def:param-space}
For each point $p_i \in \Supp(D)$, let $U_{n_i}$ be the variety parametrizing the group of units of the $n_i$-th jet $J_{p_i}^{n_i}$:
\begin{equation}\label{eq:coord param-space}
U_{n_i} := \bigl\{\, (b_{i,0}, \dots, b_{i,n_i-1}) \in \mathbb{F}_q^{n_i} : b_{i,0} \neq 0 \,\bigr\}.
\end{equation}
We define the {parameter space of trivializations} as the product variety:
\[
\mathcal{P} := \prod_{i=1}^s U_{n_i}, \quad \dim \mathcal{P} = \sum_{i=1}^s n_i = M.
\]
\end{definition}

\begin{remark}
Note that any element $f_i \in U_{n_i}$ defines a unit $f_i(t_i) \equiv \sum_{j=0}^{n_i-1} b_{i,j} t_i^j \pmod{t_i^{n_i}}$. 
\end{remark}
Given an initial trivialization $\gamma_D$ and a point $f = (f_1, \dots, f_s) \in \mathcal{P}$, we denote by $\gamma_D(f)$ the trivialization obtained by the action of the units, i.e., $\gamma_i(f) := f_i \cdot \gamma_i$. The following lemma shows that all variations of trivializations arise in this way.

\begin{lemma}\label{lm:transitivity-units}
Fix  uniformizers $t_D = (t_1, \dots, t_s)$. Let ${T}_D$ be the affine variety $\prod_{i=1}^{s} \Iso^{t_i}_{\mathrm{Tay}}(J^{n_i}_{p_i}(L),\mathbb{F}_{q}^{n_i})$. Then, for any fixed  $\gamma_D \in {T}_D$, the map induced by $\widehat{\Psi}$ (see \ref{eq:action on Taylor}))
\begin{equation*}
\begin{split}
\widehat{\Psi}_{\gamma_D}:\mathcal{P} & \longrightarrow {T}_D \\
f = (f_1, \dots, f_s) & \longmapsto  \gamma_D(f) = (f_1 \gamma_1, \dots, f_s \gamma_s)
\end{split}
\end{equation*}
is an isomorphism of affine varieties. In particular, any tuple of trivializations is uniquely determined by the action of the subgroup of rescalings on a fixed initial tuple of trivializations.
\end{lemma}

\begin{proof}
The bijectivity follows from Proposition \ref{prop:tay-simply-transitive} (the orbit maps of simply transitive group actions are bijetive). On the other hand, it follows easily from Remark \ref{rmk: Taylor algebraic group} and Remark \ref{rmk:algebraic-structure} 2), that the bijection is an isomorphism of affine varieties.
\end{proof}

Fix a basis $s_1,\dots,s_k$ of $\im(H^0(\Theta))\subseteq H^0(X,L|_D)$, and consider the corresponding differential Goppa code obtained using the fixed uniformizers $t_D$ and  trivializations $\gamma_D$, $C:=C(X,L,D,t_D,\gamma_D)\subset \F_q^M$. Now, for each $f\in\mathcal{P}$, we get a new code by using the trivializations $\gamma_{D}(f)$, which we denote by $C(f)$. 

\begin{lemma}\label{lm: regularity-G}
If $G$ is a generator matrix of $C$, then $G(f):=G\,S(f)$ is a generator matrix of $C(f)$. Here, $S(f)\in \GL_M(\F_q)$ is the block-diagonal matrix
\[
S(f)=\diag\big(S_{p_1}(\overline f_1),\dots,S_{p_s}(\overline f_s)\big),
\]
and each block $S_{p_i}(\overline f_i)\in \GL_{n_i}(\F_q)$ is the lower triangular matrix given in Lemma \ref{lm:action local parameters}. In particular, each entry of the generator matrix $G(f)$ is a regular function (in fact, a linear form) on the parameter space $\mathcal{P}$ in the variables $b_{i,j}$.
\end{lemma}
\begin{proof}
The entries of $G$ correspond to the coefficients of the Taylor expansions of  $\{s_1, \dots, s_k\}$ under the initial trivializations $\gamma_D$.
By Lemma \ref{lm:transitivity-units}, the new code $C(f)$ is obtained by applying the rescaling action of $f = (f_1, \dots, f_s) \in \mathcal{P}$, $\gamma_D(f) = (f_1 \gamma_1, \dots, f_s \gamma_s)$. In $\mathbb{F}_q^{n_i}$, the multiplication by  $f_i$  is represented by $S_{p_i}(\overline f_i) \in \text{GL}_{n_i}(\mathbb{F}_q)$.
Therefore, the new generator matrix is given by:
\[
G(f) = G \cdot \text{diag}\big(S_{p_1}(\overline f_1), \dots, S_{p_s}(\overline f_s)\big) = G \, S(f).
\]
Finally, note that the entries of the matrices $ S_{p_i}(\overline f_i)$ are linear combinations of the coefficients $b_{i,j}$. Since $G(f)=GS(f)$, each entry of $G(f)$ is a linear form on the coefficients $b_{i,j}$ as well.
\end{proof}

\begin{lemma}\label{lm:minors}
Let $C(f)\subseteq \F_q^M$ be the code generated by $G(f)$. Then $d_H(C(f))\ge d$ if and only if, for every subset $J\subseteq\{1,\dots,M\}$ with $|J|=M-d+1$, the submatrix $G(f)_J$ obtained by restricting to columns in $J$ has rank $k$.
Equivalently, for every such $J$ there exists at least one $k\times k$ minor of $G(f)_J$ which is nonzero.
\end{lemma}
\begin{proof}
This is the classical characterization
\end{proof}

\begin{theorem}\label{th:large-q-existence}
Fix $(X,L,D)$ and $d\in\mathbb{N}$ with $1\le d\le M-k+1$. Then there exists a constant $\Delta=\Delta(M,k,d)$ such that, if $q>\Delta$ the following holds: if there exists an extension field $\F_{q^m}$ and parameters $f^\star\in \mathcal P(\F_{q^m})$ such that
\[
d_H\big(C(X,L,D,t_D,\gamma_D(f^\star))\big)\ \ge\ d,
\]
then, there exists $f\in \mathcal P(\F_q)$ satisfying
\[
d_H\big(C(X,L,D,t_D,\gamma_D(f))\big)\ \ge\ d.
\]
\end{theorem}
\begin{proof}
By Lemma \ref{lm: regularity-G}, each entry of $G(f)$ is a regular function (in fact, a linear form) on the parameter space $\mathcal{P}$.
By Lemma~\ref{lm:minors}, the condition $d_H(C(f)) \ge d$ is equivalent to $\rk(G(f)_J) = k$ for every subset $J \subseteq \{1, \dots, M\}$ with $|J| = M-d+1$. By hypothesis, there exists $f^\star \in \mathcal{P}(\mathbb{F}_{q^m})$ such that this condition holds. Thus, for each $J$, there exists at least one $k \times k$ minor of $G(f^\star)_J$ which is non-zero. Let $P_J(f)$ denote such a minor; it is a polynomial of degree at most $k$ and with $M$ variables. We define
\[
P(f) := \prod_{J:\ |J|=M-d+1} P_J(f) \in \mathbb{F}_q[b_{i,j}].
\]
By construction, $P(f^\star) \neq 0$, so $P\neq 0$. Its total degree is bounded by
\[
\deg(P) \le k \cdot \binom{M}{M-d+1} = k \cdot \binom{M}{d-1} =: \Delta(M,k,d).
\]
By the Schwartz--Zippel lemma (applied to $P$), the number of zeros of $P$ in $\mathbb{F}_{q}^M$ is at most $\Delta q^{M-1}$. If $q > \Delta$, then $\Delta q^{M-1} < q^M$, ensuring the existence of $f \in \mathbb{F}_{q}^M$ such that $P(f) \neq 0$. To conclude, note that the condition $P(f) \neq 0$ already implies $b_{i,0} \neq 0$ for all $i$ (otherwise the corresponding columns in $G(f)$ would be zero, contradicting the full rank of the minors). Thus, $f$ necessarily lies in $\mathcal{P}(\mathbb{F}_q)$, and the code $C(f)$ satisfies $d_H(C(f)) \ge d$.
\end{proof}

\begin{remark}\label{rmk:large-q-comments}
Theorem~\ref{th:large-q-existence} shows that the Zariski open subset of $\mathcal{P}$ defined by $d_H\ge d$ has a rational point if it is non-empty and the base field is large enough. The bound $\Delta(M,k,d)$ given here is certainly rough and it would be interesting for practical reasons to refine it.
\end{remark}

\section{Two fundamental results}

In this final section, we present two fundamental results that highlight the importance of differential Goppa codes. First, we show that any linear block code can be realized as a differential Goppa code on a projective curve. Second, we establish that strong geometric Goppa codes, often called SAG codes in the literature, form a proper subclass of strong differential Goppa codes. Together, these findings emphasize a key advantage of our approach: the ability to construct Goppa-like codes of large length, even when working with curves that have few points, regardless of the base field size.

\subsection{Every linear block code is a differential Goppa code}

Pellikaan, Shen, and van Wee showed that every linear code can be obtained from a curve by means of Goppa's construction, thereby placing arbitrary linear codes within a geometric framework \cite{Pellikan1991}. On the other hand, related realization results in the context of convolutional codes have also been given \cite{IglesiasMunozMunozSerrano}. The following theorem may be regarded as a differential counterpart of this philosophy in the one-point setting on $\mathbb{P}^1$.

\begin{theorem}\label{th: every linear code}
Every linear code, $C$, over $\F_q$ can be expressed as a differential Goppa code $C=C(\mathbb{P}^1,L,D,\Gamma,t_D,\gamma_D)$, with $C(\mathbb{P}^1,L,D,t_D,\gamma_D)$ as in \S \ref{sec:one point} and $\Gamma\subset H^0(\mathbb{P}^1,L)$ being a suitable vector subspace.
\end{theorem}
\begin{proof}
Let $C\subseteq \F_q^n$ be a linear code of dimension $k$. Let $p_\infty\in \mathbb{P}^1$ be the point at infinity. Set $L:= O_{\mathbb P^1}(n-1)$ and $D:=n\,p_\infty$.
Choose a uniformizer $t$ at $p_\infty$ and a trivialization $\gamma$ as in the one-point construction of \S\ref{sec:one point}. 
As shown in the one-point differential Goppa construction,
$\ev_{t_D,\gamma_D}\colon H^0( \mathbb P^1,L)\longrightarrow \F_q^n$ is an isomorphism (\S\ref{sec:one point}).
Now, defining
\[
\Gamma:=\ev_{t_D,\gamma_D}^{-1}(C)\subseteq H^0(\mathbb P^1,L).
\]
we get
$C=C(\mathbb P^1,L,D,\Gamma,t_D,\gamma_D)$.
\end{proof}

\begin{remark}
\begin{enumerate}
\item The subspace \(\Gamma\subset H^0(\mathbb P^1,L)\) can be written down explicitly from the expression the matrix $G(p_{\infty})$ and a basis (that is, a generator matrix) of $C\subset\F_q^n$, say $\{c_1,\hdots,c_k\}$. A simple calculation shows that a basis of $\Gamma:=G^{-1}(C)\subset V$ is given by the polynomials
\[
f_{\ell}(t)=\sum_{i=0}^{n-1}c_{\ell,n-i}t^{i}, \quad \ell=1,\hdots, k.
\]
where $c_i=(c_{i,1},\hdots,c_{i,n})$ for each $i=1,\hdots,k$.
\item We can state the above theorem in terms of the genus zero differential Goppa codes appearing in the two-point case \S \ref{sec:two points}. Again, we can describe explicitly the subspace $\Gamma$. It is easy to see that $G(\alpha)^{-1}=G(-\alpha)$. Therefore, if a basis (that is, a generator matrix) of $C\subset\F_q^n$ is given by $\{c_1,\hdots,c_k\}$,  a basis of $\Gamma:=G^{-1}(C)\subset V$ is given by the polynomials
\[
f_\ell(t)=\sum_{m=1}^{r}\left(\sum_{j=1}^{m}\binom{m-1}{j-1}(-\alpha)^{m-j}c_{\ell,j}\right)t^{m-1}, \quad \ell=1,\hdots, k.
\]
\item The choice of \(\mathbb P^1\) is made only for simplicity. The same result can be formulated on an elliptic curve under the one-point situation (see \S \ref{sec:on point genus one}).
\end{enumerate}
\end{remark}

\subsection{Strong geometric Goppa codes form a proper subclass}

Theorem~\ref{th: every linear code} shows that allowing non-complete differential Goppa structures to represent linear codes makes the underlying geometric data to be useless. Therefore, in order for the differential Goppa structure to have a direct impact on the parameters and behavior of the code it represents, one must restrict the degrees of freedom allowed in the choice of the differential Goppa structures under consideration. 

A differential (or geometric) Goppa code $C:=C(X,L,D,t_D,\gamma_D)$ with $g=g(X)$, $d=\deg(L)$ and $n=\deg(D)$ is called strong if $2g-2<d<n$. In this case, $H^0(X,L(-D))=H^1(X,L)=0$ and Riemann-Roch theorem allows to compute the dimension of the code:
$$
\dim C=1-g+d.
$$
Strong geometric Goppa codes (also named SAG codes \cite{Pellikan1991}) are particularly interesting for many reasons. We will see now that the class of strong geometric Goppa codes form a proper class of strong differential Goppa codes. That is, there are strong differential Goppa codes that can not be described as a strong geometric Goppa code over any smooth projective curve.

\begin{proposition}
\label{prop:numerical-obstruction-Fq}
Let $q$ be a prime power, and let $n,k\in\mathbb N$ satisfy $1\le k<n$. Assume that there exists a strong geometric Goppa code (in the sense of Tsfasman--Vl\u{a}du\c{t})
$C=C(X,L,D,\overline{\gamma_D})$ over $\F_q$ such that $X$ is a smooth projective curve of genus $g$, $L$ is an invertible sheaf on $X$ of degree $d:=\deg(L)$, $D=p_1+\cdots+p_n$ is a sum of $n$ different rational and $\dim C = k$.
If we set
\[
t_q(n):=\max\!\left\{0,\left\lceil \frac{n-(q+1)}{2\sqrt q}\right\rceil\right\}.
\]
then necessarily
 $
t_q(n)\le k\le n-t_q(n).
 $
%Equivalently, if
%\[
%1\le k\le t_q(n)-1
%\qquad\text{or}\qquad
%n-t_q(n)+1\le k\le n-1,
%\]
%then no such geometric Goppa code can exist.
\end{proposition}

\begin{proof}
Since the code is strong, Riemann--Roch gives  $k=d+1-g $, hence  $d=k+g-1 $.
From  $2g-2<d=k+g-1 $ we get   $g\le k $.
Since  $d=k+g-1<n $ it holds $g<n-k+1 $, hence  $g\le n-k $.
Therefore  
$$g\le \min(k,n-k).$$ 
On the other hand, since $D$ is supported on $n$ distinct rational points,  $n\le \#X(\F_q) $.
By the Hasse--Weil bound (see \cite{Weil1949,Stichtenoth, Ihara1981}),  $\#X(\F_q)\le q+1+2g\sqrt q $.
Thus  $n\le q+1+2g\sqrt q $, so
 $
g\ge \frac{n-(q+1)}{2\sqrt q}.
 $
Since $g$ is an integer,
\[
g\ge \left\lceil \frac{n-(q+1)}{2\sqrt q}\right\rceil.
\]
If this lower bound is negative, it is vacuous, so the effective lower bound is
\[
g\ge t_q(n):=\max\!\left\{0,\left\lceil \frac{n-(q+1)}{2\sqrt q}\right\rceil\right\}.
\]
Combining both inequalities for $g$, we obtain
 $
t_q(n)\le \min(k,n-k),
 $
which is equivalent to
 $
t_q(n)\le k\le n-t_q(n).
 $
This proves the claim.
\end{proof}

\begin{theorem}
\label{thm:proper-subclass-strong-Fq}
Let $q$ be a prime power, and let
\[
t_q(n):=\max\!\left\{0,\left\lceil \frac{n-(q+1)}{2\sqrt q}\right\rceil\right\}.
\]
Let $n,k$ be integers with $1\le k<n$ and assume that $ n\ge q+2+\lfloor 2\sqrt q\rfloor$.
If
\[
1\le k\le t_q(n)-1
\qquad\text{or}\qquad
n-t_q(n)+1\le k\le n-1.
\]
then there exists a strong differential Goppa code over $\F_q$ of length $n$ and dimension $k$ which is not a strong geometric Goppa code over $\F_q$.

In particular, over every finite field $\F_q$, the class of strong geometric Goppa codes is a proper subclass of the class of strong differential Goppa codes.
\end{theorem}

\begin{proof}
Set $X=\mathbb P^1_{\F_q}$
and let $p_\infty\in X(\F_q)$ be the point at infinity. Set  $d:=k-1 $.
Since $1\le k<n$, we have  $0\le d<n $.
Now consider the one-point differential Goppa code $C=C(X,L,D,t_D,\gamma_D)$
defined by $L=\mathcal O_{\mathbb P^1}(d)$ and $D=n\,p_\infty$. Since  $g(X)=0 $, the strong condition is
$-2<d<n$,
which holds true. Hence $C$ is a strong differential Goppa code. By Riemann--Roch,
 $
\dim C=1-g+d=1+0+(k-1)=k.
 $
So $C$ has length $n$ and dimension $k$.

Suppose, for contradiction, that $C$ were a strong geometric Goppa code over $\F_q$. Since $C$ has length $n$ and dimension $k$, Proposition~\ref{prop:numerical-obstruction-Fq} would force
 $
t_q(n)\le k\le n-t_q(n).
 $
This contradicts the assumption on $k$. Hence $C$ is not a strong geometric Goppa code.

It remains to show that, for each finite field $\F_q$, the hypothesis of the theorem is nonempty for at least one pair $(n,k)$. Choose an integer $n$ such that $n\ge q+2+\lfloor 2\sqrt q\rfloor$. Then  $t_q(n)\ge 2 $. In particular, $1\le t_q(n)-1$, so the choice  $k=1 $ satisfies $1\le k\le t_q(n)-1$.
Hence the pair  $(n,1) $ is in the range covered by the theorem. Now apply the first part of the proof to this specific pair  $(n,1) $. The construction on  $\mathbb P^1_{\F_q} $ yields a strong differential Goppa code of length  $n $ and dimension  $1 $. By the numerical obstruction in Proposition~\ref{prop:numerical-obstruction-Fq}, no strong geometric Goppa code over  $\F_q $ can have these parameters. Therefore this code is not a strong geometric Goppa code. It follows that for every $\F_q$ there exists a strong differential Goppa code which is not a strong geometric Goppa code. 
\end{proof}

%-------------------------------------------------
\appendix

\section{Jets, higher order derivations and differential calculus}\label{app: jets}

Throughout this appendix, $X$ denotes a smooth projective curve over a field
$\mathbb{F}_{q}$. Most of the constructions and results presented here remain valid
for general separated schemes, possibly with minor modifications. Although the standard reference for jet bundles is \cite{EGAIV}, all the results given in the appendix can be found in \cite{Laksov1984,LaksovThorup1991,Laksov1994,PerkinsonCurves,Piene1977} in one way or another.

Let $X$ be a smooth projective curve over a field $\mathbb{F}_{q}$. The diagonal immersion 
$\Delta:X\hookrightarrow X\times X$ is closed and we denote by $I$ the ideal of $X$ in $X\times X$. Let $\pi_{1},\pi_{2}:X\times X\rightarrow X$  be the projections onto the first and the second factor respectively. 

\subsection{The trivial case. Rational functions}\label{sec:ppal_parts_ring}
\begin{definition}
Let $X$ be a smooth projective curve over a field $\mathbb{F}_{q}$.
For each natural number $i\in\mathbb{N}$  the sheaf of jets of $X$ order $i$ is the the sheaf of rings
$$
J^i_{X}=\pi_{1*}\left(O_{X\times X}/I^{i+1}\right).
$$
\end{definition}
The sheaf $J^{i}_X$ admits two different $O_X$-module (therefore, $O_X$-algebra) structures. Recall that for $l=1,2$, the structure morphism of schemes $\pi_l$ induces a homomorphism of sheaves of rings on $X$, denoted by $\pi_l^\sharp: O_X \to \pi_{l*} O_{X\times X}$.
\begin{enumerate}
    \item
    Consider the composition of $\pi_1^\sharp$ with the pushforward of the quotient map $q: O_{X\times X} \to O_{X\times X}/I^{i+1}$:
    $$ 
    \iota^{i}_X: O_X \xrightarrow{\pi_1^\sharp} \pi_{1*} O_{X\times X} \xrightarrow{\pi_{1*}(q)} \pi_{1*}(O_{X\times X}/I^{i+1}) = J^i_X.
    $$
    This ring homomorphism $\iota^{i}_{X}$ defines the first $O_X$-module structure (the standard vector bundle structure). Locally, it maps $f \mapsto f \otimes 1 \pmod{I^{i+1}}$.
    \item
    Similarly, we have the map $\pi_2^\sharp: O_X \to \pi_{2*} O_{X\times X}$. Since the sheaf $O_{X\times X}/I^{i+1}$ is supported on the diagonal $\Delta(X)$, and $\pi_1|_{\Delta(X)} = \pi_2|_{\Delta(X)}$, there is a canonical isomorphism of sheaves of rings $\phi: \pi_{2*}(O_{X\times X}/I^{i+1}) \xrightarrow{\sim} \pi_{1*}(O_{X\times X}/I^{i+1})$. We define the second map $d^i$ as the composition:
    $$ 
    d^i_X: O_X \xrightarrow{\pi_2^\sharp} \pi_{2*} O_{X\times X} \xrightarrow{\pi_{2*}(q)} \pi_{2*}(O_{X\times X}/I^{i+1}) \xrightarrow{\phi} J^i_X.
    $$
    Locally, it maps $f \mapsto 1 \otimes f \pmod{I^{i+1}}$ and it defines the second action.  The map $d^i_X$ is a morphism of $\mathbb{F}_{q}$-algebras.
\end{enumerate}
We will always be considering the $O_X$-algebra structure on $J^{i}_X$ given by $\iota^{i}_X$.

\subsubsection{{Fundamental properties}} Assume now that $X$ is a smooth projective curve over $\mathbb{F}_{q}$. The sheaves $J^{i}_{X}$ enjoy the following fundamental properties:
\begin{enumerate}
\item There is a canonical short exact sequence of $O_X$-modules for each $i\in\mathbb{N}$, $$0\rightarrow \omega_X^{\otimes i} \rightarrow J^{i}_{X}\rightarrow J^{i-1}_{X}\rightarrow 0.$$
\item If $U\subset X$ is an open subset, then $J^{i}_{X}|_{U}\simeq J^{i}_{U}$  canonically.
\item If $x\in X$ is a rational point, then $J^{i}_{X}|_{x}\simeq O_{X,x}/\mathfrak{m}^{i+1}O_{X,x}$ canonically. 
\item If $U=\textrm{Spec}(A)$ is an affine open subset, then $$J^{i}_{X}|_{U}=(A\otimes_k A)/I_U^{i+1}\simeq O_{U}^{i+1}, \ \ I_U:=\Ker(A\otimes_k A\rightarrow A),$$ 
non-canonically. A basis for $J^{i}_{X}|_{U}$ can be found as follows. Let $z\in A$ be a generator of $\omega_{X}|_U=I_U/I_U^2$. Let us denote by $\delta^{i}_{X}=d^{i}_{X}-\iota^{i}_{X}$. Define 
$$\xi:=\delta^{i}_{U}(z)=(z\otimes 1- 1\otimes z) \pmod{I^{i+1}}\in J^{i}_{X}|_{U}.$$ 
Then, $\{1,\xi,\xi^2,\hdots,\xi^{i}\}$ is a basis of $J^{i}_{X}|_{U}$ as a $O_X$-module.
\end{enumerate}
\begin{remark}\label{rmk:uniformizer basis}
For a rational point $p\in U$, the fiber $J_p^n=(J_X^n)_p\otimes_{O_{X,p}}k(p)$ is obtained by evaluating the first factor at $p$, so $\xi$ maps to
$z(p)-z=-(z-z(p))\in \mathfrak m_p/\mathfrak m_p^{n+1}$.
Since $dz(p)\neq 0$, the element $t:=z-z(p)$ is a uniformizer of $p$, and the basis $\{1,\xi,\xi^2,\hdots,\xi^{i}\}$ specializes to the basis
$\{1,t,\dots,t^n\}$ of $J_p^n$ (up to a sign).
\end{remark}

\subsubsection{{Higher derivatives and Taylor expansions of rational functions via jets}}\label{par:calculus}
Let $\mathbb{F}_{q}$ be a ring, let $A$ and $B$ $\mathbb{F}_{q}$-algebras, and  $m \in \mathbb{N} \cup \{\infty\}$. 
\begin{definition}
A {higher derivation of order $m$} from $A$ to $B$ over $\mathbb{F}_{q}$ is a sequence $(D_0, \dots, D_m)$, where $D_0 \colon A \to B$ is an $\mathbb{F}_{q}$-algebra homomorphism and $D_1, \dots, D_m \colon A \to B$ are $\mathbb{F}_{q}$-module homomorphisms such that for all $x, y \in A$ and all $l = 0, \dots, m$:
    \begin{equation}
        D_l(xy) = \sum_{i+j=l} D_i(x)D_j(y) \ \ \textrm{(Leibniz rule) }.
    \end{equation}
\end{definition}
We denote by $\textrm{Der}^{m}_{k}(A,B)$ the set of higher derivations of order $m$ from $A$ to $B$.

The $\mathbb{F}_{q}$-linear map $d^{i}_{X}$ is intimately related to higher order derivations of functions on the curve in the following sense. We restrict ourselves to an affine open subset $U=\textrm{Spec}(A)$ of the smooth projective curve $X$  and we pick $z\in A$ such that $\xi:=\delta^{i}_{U}(z)$ generates  $\omega_{X}|_U$. Then, given $a\in A$, we have 
$$
d^{i}_{U}(a)=D_{U,z}^{0}(a)+D_{U,z}^{1}(a)\xi+\cdots +D_{U,z}^{i}(a)\xi^{i}
$$
where 
\begin{equation}\label{eq:higher derivative of function}
D_{U,z}^{j}(a) = \left[ 
\begin{array}{l}
\textrm{$j$-th coefficient of $d^{i}_{U}(a)$} \\
\textrm{in the basis $\{1,\xi,\hdots,\xi^{i}\}$}
\end{array}
\right].
\end{equation}
Since $d^{i}_{U}$ is an $\mathbb{F}_{q}$-algebra homomorphism, it follows that each $D_{U,z}^{j}$ determines map $D_{U,z}^{j}:A\rightarrow A$ which is $\mathbb{F}_{q}$-linear and satisfies
$$
D^{j}_{U,z}(ab)=\sum_{l=0}^{j}D_{U,z}^{l}(a)\,D_{U,z}^{j-l}(b),
\qquad j=0,\dots,i.
$$
That is, $(D_{U,z}^{0},\hdots,D_{U,z}^{i})\in\textrm{Der}^{i}_{k}(A,A)$. There is an important fact to highlight: each $D_{U,z}^{j}$ above is a differential operator in the sense of Grothendieck; that is, each $D_{U,z}^{j}:A\rightarrow A$ is equal to $pr_j \circ d_{U}^{i}$ where $pr_j:J^{i}_{U}\rightarrow A$ is the projection onto the $j$-th factor. In particular, $D_{U,z}^{1}$ is an ordinary derivation.

Finally, note that the canonical surjections $J^{i}_{X}\rightarrow J^{i-1}_{X}$ determine a projective system and we may define 
\begin{equation*}
J^{\infty}_{X}:=\varprojlim_i J_{X}^{i}.
\end{equation*}
The canonical maps $d_X^{i}$ induce a map
$$
d^{\infty}_X: O_X\rightarrow  J^{\infty}_{X}
$$
which is locally described by the full Taylor sequence. More precisely, given an affine open subset $U=\textrm{Spec}(A)$ and $z\in A$ such that $\xi:=\delta^{i}_{U}(z)$ generates  $\omega_{X}|_U$, then
$$
d^{\infty}_{U}(a)=D_{U,z}^{0}(a)+D_{U,z}^{1}(a)\xi+\cdots +D_{U,z}^{i}(a)\xi^{i}+\cdots
$$

\subsubsection{Comparison with the classical Taylor expansion}
We refer the reader to \cite[\S 4.2]{Stichtenoth} for a detailed exposition of Taylor expansions in function fields.

Let $x: K\rightarrow \mathbb{Z}\cup\{\infty\}$, be a discrete $\mathbb{F}_{q}$-valuation and $O_{x}\subset K$ the corresponding discrete valuation ring. The $\mathfrak{m}_{x}$-adic completion is defined by
$$
\widehat{O}_{x}:=\varprojlim_{n}O_{x}/\mathfrak{m}_{x}^n
$$
Clearly, the canonical morphism of $O_x$-modules $O_{x}\rightarrow\widehat{O}_{x}$ is injective. Therefore, the induced morphism between their fields of fractions
$$
\mathfrak{l}_{t_x}:K\hookrightarrow K_{x}
$$
is injective as well. Given a parameter at $x$, $t_x\in\mathfrak{m}_{x}\setminus\mathfrak{m}_{x}^2$, it is well known that $K_{x}\simeq \mathbb{F}_{q}((t_x))$. Thus, given $f\in K$, we define the Laurent expansion of $f$ at $x$ respect to $t_x$ as
\begin{equation}\label{eq:laurent expansion}
f(t_x):=\mathfrak{l}_{t_x}(f)=\sum_{n\in\mathbb{Z}}f_{n}t_{x}^{n}\in \mathbb{F}_{q}((t_x)).
\end{equation}
It is easy to see that the order of $f$ at $x$ satisfies:
$$
\textrm{ord}_x(f)=\textrm{min}\{n\in\mathbb{N}| \ f_n\neq 0\}
$$
If $f$ is regular at $x$, that is, $f\in O_{x}$, then
$$
f(t_x)=\sum_{n=0}^{\infty}f_{n}t_{x}^{n}\in \mathbb{F}_{q}[[t_x]]\footnote{Note that, fixed $t_x$, threre are canonical isomorphisms $\mathbb{F}_{q}[[t_x]]/t_{x}^{i+1}\simeq \bigoplus_{n=0}^{i}\mathfrak{m}_{x}^{n}/\mathfrak{m}_{x}^{n+1}$ and, therefore, a canonical isomorphism $\mathbb{F}_{q}[[t_x]]\simeq \varprojlim_i \bigoplus_{n=0}^{i}\mathfrak{m}_{x}^{n}/\mathfrak{m}_{x}^{n+1}$. This can be used to compute the derivatives.}
$$
and we define the {derivative of order $n$ of $f$ at $x$ respect to $t_x$} as
\begin{equation}\label{eq:derivative at a point}
D_{t_{x}}^{n}(f):=f_n.
\end{equation}
In order to justify this definition, note that, $\mathfrak{l}_{t_x}$ being a morphism of $\mathbb{F}_{q}$-algebras, it holds
\begin{enumerate}
\item $D_{t_{x}}^{0}(f)=f(x)$ for all $f\in O_x$\footnote{Note that, $O_x$ being discrete $\mathbb{F}_{q}$-valuation ring imples, in particular, that $\mathbb{F}_{q}\subset O_x$.Therefore, the composition $\mathbb{F}_{q}\subset O_x\rightarrow O_x/\mathfrak{m}_x$ determines an isomorphism $\mathbb{F}_{q}\simeq O_x/\mathfrak{m}_x$ for every rational point $x\in X$.}.
\item $D_{t_{x}}^{n}(\lambda f+\mu g)=\lambda D_{t_{x}}^{n}(f)+\mu D_{t_{x}}^{n}(g)$ for all $\lambda,\mu\in\mathbb{F}_{q}$ and $f,g\in O_x$.
\item (Leibniz rule) $D_{t_{x}}^{n}(f g)=\sum_{i=0}^{n}D_{t_{x}}^{i}(f) D_{t_{x}}^{n-i}(g)$ for all $f,g\in O_{x}$.
\end{enumerate}
As we have seen in \ref{par:calculus}, each truncated tuple $(D_{t_{x}}^{0},D_{t_{x}}^{1},\hdots, D_{t_{x}}^{n})$ determines a higher order derivation. In particular, $D_{t_{x}}^{1}$ is an ordinary derivation. We will also use the classical notation 
\begin{equation}\label{eq:classical derivative}
f'(t_x)=D_{t_{x}}^{1}(f).
\end{equation}
Finally, note that if $f$ has a pole at $x$ of order $n$ then, for each $n\geq i>0$, we have
\begin{equation}\label{eq:negative coeff derivatives}
a_{-i}=D_{t_x}^{n-i}(t_x^n f)
\end{equation}

\subsection{The general case. Meromorphic sections}
Let $X$ be a  smooth projective curve over a field $\mathbb{F}_{q}$.
Consider now an invertible sheaf $L$ on $X$. Then, the sheaf of jets of $L$ of order $i$ is defined as
\begin{equation}\label{eq:n-jets}
J^i_{X}(L):=\pi_{1*}\left(\pi_{2}^{*}L\otimes_{O_{X\times X}}O_{X\times X}/I^{i+1}\right).
\end{equation}
This sheaf carries a natural structure of module over the algebra of jets $P^i_X := J^i_X(O_X)$. Consequently, the two $O_X$-module structures on $P^i_X$ induce two compatible structures on $J^i_{X}(L)$. 

Unless stated otherwise, we will consider on $J^i_{X}(L)$ the $O_X$-module structure deduced from the left homomorphism $\iota^{i}: O_X \to P^i_X$. \subsubsection{{Fundamental properties}:} 
Assume now that $X$ is a smooth projective curve over $\mathbb{F}_{q}$.
The sheaves $J^{i}_{X}(L)$ enjoy the fundamental properties analogous to $J^{i}_X$:
\begin{enumerate}
\item There is a canonical short exact sequence of $O_X$-modules, 
\begin{equation}\label{eq:fundamental exact sequence}
0\rightarrow \omega_X^{\otimes i}\otimes L \rightarrow J^{i}_{X}(L)\rightarrow J^{i-1}_{X}(L)\rightarrow 0.
\end{equation}
\item If $U\subset X$ is an open subset, then $J^{i}_{X}(L)|_{U}\simeq J^{i}_{U}(L|_U)$  canonically.
\item If $x\in X$ is a rational point, then $J^{i}_{X}(L)|_{x}\simeq L_{x}/\mathfrak{m}^{i+1}L_{x}$ canonically. 
\item If $U=\textrm{Spec}(A)$ is an affine open subset trivializing $L$ and $\gamma:L|_U\simeq O_X|_U$ is a trivialization, then 
$$J^{i}_{X}(L)|_{U}\simeq J^{i}_{X}(L|_{U})\overset{J^{i}(\gamma)}{\simeq} J^{i}_{X}|_{U} \simeq O_{U}^{i+1}$$ 
non-canonically. This isomorphism will be denoted by $J^{i}(\gamma)$ as well.
\item The canonical surjection $\pi_2^*L\rightarrow\pi_2^*L\otimes O_{X\times X}/I^{i+1}$ induces by adjunction a morphism of $O_X$-modules $L\rightarrow \pi_{2*}(\pi_2^{*} L\otimes O_{X\times X}/I^{i+1})$. Since $\pi_2^{*} L\otimes O_{X\times X}/I^{i+1}$ is supported in the diagonal, there is a canonical $\mathbb{F}_{q}$-module isomorphism $ \pi_{2*}(\pi_2^{*} L\otimes O_{X\times X}/I^{i+1})\simeq  \pi_{1*}(\pi_2^{*} L\otimes O_{X\times X}/I^{i+1})$. All together gives a $\mathbb{F}_{q}$-linear map
$$
d^{i}_{L}:=L\rightarrow J^{i}_{X}(L).
$$
\end{enumerate}
\begin{definition}\label{def:truncated Taylor map}
The $i$-th truncated Taylor sequence map of $L$ is the $\mathbb{F}_{q}$-linear map $d^{i}_{L}$ defined above. 
\end{definition}
It is important to hilight that the canonical isomorphism mentioned above,
\[
J_X^{i}(L)\big|_{x} \;\simeq\; L_x/\mathfrak{m}_x^{i+1}L_x = L/L(-(i+1)x),
\]
makes the truncated Taylor sequence map $d_L^{i}:L\to J_X^{i}(L)$ composed with the restriction morphism $J_X^{i}(L)\rightarrow J_X^{i}(L)|_x$ to coincide with the natural quotient map
$$L\to L/L(-(i+1)x).$$

\subsubsection{{Derivatives of sections of invertible sheaves}}\label{sec: variations}

A key property of the truncated Taylor sequence map is that for every open subset, $U\subset X$, we have a commutative square 
\begin{equation}\label{eq: Derivatives of sections of invertible sheaves}
\xymatrix{
L|_U \ar[r]^{d^{i}_{L}|_U} \ar[d]_{\gamma} & J^{i}(L)|_U  \ar[r]^{\textrm{canonical}} & J^{i}(L|_U) \ar[d]^{J^{i}(\gamma)} \\
O_U \ar[rr]_{d^{i}_{U}} & &  J^{i}_U 
}
\end{equation}
Therefore, given a local section $s\in L|_U$, and a generator, $z\in\omega_{X}|_U$, we have
$$
d^{i}_{L}|_U(s)=D_{U,z}^{0}(\gamma(s))+D_{U,z}^{1}(\gamma(s))\overline{\xi}+\cdots +D_{U,z}^{i}(\gamma(s))\overline{\xi^{i}}
$$
with $\overline{\xi^j}:=J^{i}(\gamma)^{-1}(\xi^{j})$. 
\begin{definition}\label{def: derivative of sections}
We define the $j$-th derivative of $s$ with respect to the trivialization $(U,\gamma)$ and the parameter $z$ as
\begin{equation}\label{eq: derivative of local section}
D^{j}_{U,\gamma,z}(s):=D^{j}_{U,z}(\gamma(s))
\end{equation}
\end{definition}

The dependency of $D^{j}_{U,\gamma,z}$ with respect to $U$, $\gamma$ and $z$ (assuming $U=\textrm{Spec}(A)$) can be described as follows:
\begin{enumerate}
\item Change of trivialization: Let $\gamma':L|_{U}\simeq O_U$ be another trivialization. Then, there is a rational function $f\in O_U$ such that $\gamma'=f\gamma$. So,
\begin{equation}\label{eq: trivialization derivation}
D^j_{U,\gamma',z}(s)=D^j_{U,z}(\gamma'(s))=D^j_{U,z}(f\gamma(s))=\sum_{t=0}^{j}D^{j-t}_{U,z}(f)D^t_{U,z,\gamma}(s)
\end{equation}
\item Change of parameter (Faà di Bruno Formula): Fix now another generator $w\in A$ of $\omega_X|_U$ and set $\theta:=\delta_{U}^{i}(w)\in J^{i}_U$. With respect to the basis $\{1,\theta,\hdots,\theta^{i}\}$, we have 
$$\xi=\sum_{l=0}^{i}D_{U,w}^{l}(z)\theta^{l}.$$ 
On the other hand, the two different expressions of $d^{i}_{L}|_U(s)$ must coincide. That is, we have the equality
$$
\sum_{j=0}^{i}D^j_{U,z,\gamma}(s)\xi^j=\sum_{j=0}^{i}D^j_{U,w,\gamma}(s)\theta^j
$$
Combining the above equalities yields, for $m=0,\dots,i$,
\begin{equation}\label{eq: faa di bruno}
\begin{split}
D^{m}_{U,w,\gamma}(s)
&=
\sum_{j=0}^{m}
D^{j}_{U,z,\gamma}(s)
\left(
\sum_{\substack{l_1+\cdots+l_j=m \\ 1 \le l_\alpha \le m}}
\prod_{\alpha=1}^{j} D^{\,l_\alpha}_{U,w}(z)
\right)=\\
&=
\sum_{j=0}^{m}
D^{j}_{U,z,\gamma}(s)
\left(
\sum_{\substack{(n_1,\dots,n_m)\in\mathbb{N}^m\\
\sum_{r=1}^m r n_r = m\\[2pt]
\sum_{r=1}^m n_r = j}}
\binom{j}{n_1,n_2,\dots,n_m}
%\frac{j!}{\prod_{r=1}^i n_r!}\;
\prod_{r=1}^m \big(D^{r}_{U,w}(z)\big)^{n_r}
\right).
\end{split}
\end{equation}
\item Change of open subset: Regarding the dependency on $U$ it is worth to point out that $D^{j}_{U,\gamma,z}(s)\in A$ and if $V=\textrm{Spec}(B)\subset U$ (here $B=A_V$), then $D^{j}_{V,\gamma,z}(s|_U)$ is nothing but the localization
\begin{equation}\label{eq: localization of derivation}
D^{j}_{V,\gamma,z}(s|_U)=\dfrac{D^{j}_{U,\gamma,z}(s)}{1}\in B,
\end{equation}
A particularly important consequence is that, given a rational point $x\in U\subset X$, the evaluation at $x$ of the rational function $D_{U,\gamma,z}^{j}(s)$ is obviously
$$
D_{U,\gamma,z}^{j}(s)(x)=
\left[
\begin{array}{l}
\textrm{coefficient of order $j$ of the Taylor sequence of $\gamma(s)$}\\
\textrm{at $x\in X$ respect to the uniformizer $t_x:=\dfrac{z}{1}\in O_x$}
\end{array}
\right].
$$
That is, 
\begin{equation}\label{eq: coincidence derivatives}
D_{U,\gamma,z}^{j}(s)(x)=D^n_{t_x}(\gamma(s))
\end{equation} 
as defined in \ref{eq:derivative at a point}.
\end{enumerate}

\begin{example}
Let $\mathbb{F}_{q},K,X$ be as above, $G=\sum_{x\in X}n_{x}\cdot x$ a divisor of $X$ and $O_X(G)$ the invertible sheaf associated to $G$.
Let us see now how to trivialize $O_X(G)$ around a rational points $x\in X$. We distinguish two situations:
\begin{enumerate}
\item If $x\notin \textrm{supp}(G)$ and we set $U_x:=X\setminus \textrm{supp}(G)$, then $O_X(G)(U_x)=\{f\in K| \ (f)_{U_x}\geq 0\}\cup\{0\}= {O}_{X}(U_x)$. Thus, $O_X(G)|_{U_{x}}$ is indeed trivial.
\item If $x\in\textrm{supp}(G)$, let $t_{x}\in K$ be a uniformizer of the point $x$. Let us set $U_{x}=X\setminus\{\textrm{poles of $t_{x}$}\}$. For any $f\in O_X(G)(U_{x})$, consider the element $t_{x}^{n_{x}}\cdot f\in K$. Note that $(t_{x}^{n_{x}}\cdot f)\geq 0$, so $t_{x}^{n_{x}}\cdot f\in {O}_{X}(U_{x})$. The morphism 
\begin{equation}\label{eq: trivialization invertible sheaf}
\gamma: O_X(G)(U_{x})\rightarrow {O}_{X}(U_{x})
\end{equation}
given by $f\mapsto t_{x}^{n_{x}}\cdot f$ is an isomorphism, that is, a trivialization over $U_{x}$.
\end{enumerate}
Assume we are at the second situation. Set $V_x:=U_x\setminus \{\textrm{ zeroes of $t_x$} \}$. Then, $\gamma:O_{X}(G)(V_x)\simeq O_{X}(V_x)$ is a trivialization and $dt_x$ generates $\omega_X(V_x)$. Then, the value at $x$ of the derivative $D_{U,\gamma,z}^{j}(f)$ is just $D^n_{t_x}(t_x^{n_x}f)$.
\end{example}

\bibliographystyle{amsplain}

\bibliography{biblio}

\end{document}